\documentclass[12pt,leqno]{article}

\usepackage{amsmath,amssymb,amsthm}
\usepackage[all]{xy}
\usepackage{lscape}

\makeatletter

\newtheorem{theorem}{Theorem}[section]
\newtheorem{proposition}[theorem]{Proposition}
\newtheorem{lemma}[theorem]{Lemma}
\newtheorem{corollary}[theorem]{Corollary}

\theoremstyle{definition}

\newtheorem{example}[theorem]{Example}

\newtheorem{question}[theorem]{Question}

\theoremstyle{remark}

\theoremstyle{remark}
\newtheorem{remark}[theorem]{Remark}

\def\({{\rm (}}
\def\){{\rm )}}

\let\Mathrm\operator@font
\let\Cal\mathcal
\let\Bbb\mathbb
\newcommand{\fm}{\ensuremath{\mathfrak m}}

\newcommand{\fp}{\ensuremath{\mathfrak p}}

\def\standop#1{\mathop{\Mathrm #1}\nolimits}
\def\difstop#1#2{\expandafter\def\csname #1\endcsname{\standop{#2}}}
\def\defstop#1{\difstop{#1}{#1}}

\defstop{AB}
\defstop{ann}
\defstop{Ass}
\defstop{Aut}
\defstop{Alggrp}

\defstop{CMFI}
\defstop{codim}
\defstop{Coh}
\defstop{Coker}
\defstop{Cone}
\defstop{Cl}
\defstop{Cox}
\difstop{charac}{char}

\defstop{depth}
\difstop{divisor}{div}
\defstop{Div}
\defstop{DF}

\defstop{EM}
\defstop{embdim}
\defstop{End}
\defstop{ev}
\defstop{Ext}

\defstop{Flat}
\defstop{Func}

\defstop{Gal}
\def\GL{\text{\sl{GL}}}
\def\gr{_{\Mathrm{gr}}}

\difstop{height}{ht}
\defstop{Hom}

\def\id{\mathord{\Mathrm{id}}}

\difstop{Image}{Im}
\defstop{ind}

\defstop{Ker}

\defstop{Lch}
\defstop{length}
\defstop{Lin}
\defstop{Lqc}
\defstop{lqc}

\defstop{Mat}
\defstop{Max}
\defstop{Min}
\defstop{Mod}
\defstop{Mor}
\defstop{MCM}
\def\MM#1{{}_{#1}\Bbb M}

\defstop{Nerve}
\defstop{NonSFR}
\defstop{NonCMFI}
\defstop{NonNor}
\defstop{Nor}
\def\nz{^\star}
\def\sep{_{\Mathrm{sep}}}

\defstop{PA}
\defstop{Prin}
\defstop{PM}
\defstop{Proj}
\defstop{Pic}

\defstop{Qch}
\defstop{qch}

\defstop{rad}
\defstop{rank}
\def\red{_{\Mathrm{red}}}
\defstop{res}
\defstop{Reg}

\let\specialarrow\xrightarrow

\defstop{Spec}
\defstop{supp}
\defstop{Supp}
\defstop{Sym}
\defstop{Sing}
\defstop{SFR}
\defstop{Soc}

\difstop{tdeg}{tdeg}

\defstop{Tor}
\difstop{trace}{tr}

\defstop{Zar}

\def\I{\Cal I}
\def\R{\Cal R}
\def\O{\Cal O}

\def\P{\Cal P}

\def\fm{\mathfrak{m}}
\def\fp{\mathfrak{p}}

\let\indlim\varinjlim
\let\projlim\varprojlim

\def\sdarrow#1{\downarrow\hbox to 0pt{\scriptsize$#1$\hss}}
\def\suarrow#1{\uparrow\hbox to 0pt{\scriptsize$#1$\hss}}
\def\ssearrow#1{\searrow\hbox to 0pt{\scriptsize$#1$\hss}}

\def\ru#1{\lceil #1 \rceil}

\def\section{\@startsection{section}{1}{\z@ }%
{-3.5ex plus -1ex minus -.2ex}{2.3ex plus .2ex}{\bf }}

\long\def\refname{\par\kern -3ex
\begin{center}\rm R\sc{eferences}\end{center}\par\kern 
-2ex}

\def\@seccntformat#1{\csname the#1\endcsname.\quad}

\def\@@@sect#1#2#3#4#5#6[#7]#8{%
   \ifnum #2>\c@secnumdepth 
      \def \@svsec {}\else \refstepcounter {#1}%
      \def\@svsec{}
   \fi 
   \@tempskipa #5\relax 
   \ifdim \@tempskipa >\z@ 
     \begingroup #6\relax \@hangfrom {\hskip #3\relax 
     \@svsec}{\interlinepenalty \@M #8\par }\endgroup 
     \csname #1mark\endcsname {#7}
   \else 
   \def \@svsechd {#6\hskip #3\@svsec #8\csname #1mark\endcsname {#7}}
   \fi \@xsect {#5}}

\def\@@@startsection#1#2#3#4#5#6{%
 \if@noskipsec \leavevmode \fi \par \@tempskipa #4\relax \@afterindenttrue 
 \ifdim \@tempskipa <\z@ \@tempskipa -\@tempskipa \@afterindentfalse 
 \fi \if@nobreak \everypar {}\else \addpenalty {\@secpenalty }\addvspace 
  {\@tempskipa }\fi \@ifstar {\@ssect {#3}{#4}{#5}{#6}}{\@dblarg 
  {\@@@sect {#1}{#2}{#3}{#4}{#5}{#6}}}}

\def\theparagraph{\thesection.\arabic{paragraph}}
\def\aparagraph{\@@@startsection{paragraph}{2}{\z@ }%
              {1.75ex plus .2ex minus .15ex}{-1em}{\bf(\theparagraph) } }
\def\paragraph{\@@@startsection{paragraph}{2}{\z@ }%
              {1.75ex plus .2ex minus .15ex}{-1em}{}{\bf(\theparagraph)} }

\let\c@theorem\c@paragraph

\title{Equivariant total ring of fractions and 
factoriality of rings generated by semiinvariants}
\author{M{\sc itsuyasu} H{\sc ashimoto}}
\date{\normalsize
Graduate School of Mathematics, Nagoya University\\
Chikusa-ku,  Nagoya 464--8602 JAPAN\\
{\small \tt hasimoto@math.nagoya-u.ac.jp}}

\begin{document}

\maketitle
\footnote[0]
    {2010 \textit{Mathematics Subject Classification}. 
    Primary 13A50, 
    Secondary 13C20.
    Key Words and Phrases.
    invariant subring, UFD, character group.
}

\begin{abstract}
Let $F$ be an affine flat group scheme over a commutative ring $R$,
and $S$ an $F$-algebra (an $R$-algebra on which $F$ acts).
We define an equivariant analogue $Q_F(S)$ of the total ring of fractions
$Q(S)$ of $S$.
It is the largest $F$-algebra $T$ such that $S\subset T\subset Q(S)$, and
$S$ is an $F$-subalgebra of $T$.
We study some basic properties.

Utilizing this machinery, we give some new criteria for factoriality
(UFD property) of (semi-)invariant 
subrings under the action of affine algebraic groups,
generalizing a result of Popov.
We also prove some variations of classical results on factoriality of 
(semi-)invariant subrings.
Some results over an algebraically closed base field are generalized to those
over an arbitrary base field.
\end{abstract}

\section{Introduction}

Throughout this paper, $k$ denotes a field, and $G$ denotes an
affine smooth algebraic group over $k$.

Let $S$ be a $G$-algebra.
That is, a $k$-algebra with a $G$-action.
Study of ring theoretic properties of the invariant subring $S^G$ is
an important part of invariant theory.
In this paper, we discuss the factoriality of $S^G$.

Popov \cite[p.~376]{Popov2} remarked the following:

\begin{theorem}[Popov]\label{popov.thm}
Let $k$ be algebraically closed.
If 
\begin{description}
\item[(i)] $S$ is a UFD;
\item[(ii)] the character group $X(G)$ of $G$ is trivial; and
\item[(iii)] One of the following hold:
\begin{description}
\item[(a)] $S$ is finitely generated and $G$ is connected; or
\item[(b)] $S^\times\subset S^G$.
\end{description}
\end{description}
Then $S^G$ is a UFD.
\end{theorem}

Some variation of the theorem for the case that {\bf (b)} is assumed is 
treated in \cite{Hochster}.
The case that $G$ is a finite group and $S$ is a polynomial ring is
found in \cite[(1.5.7)]{Smith}.

Our main objective is to generalize this theorem, focusing the case that
{\bf (a)} is assumed.
Our main theorem is the following.

\begin{trivlist}
\item[\bf Theorem~\ref{main.thm}]\it
Let $S$ be a finitely generated $G$-algebra which is a normal domain.
Assume that $G$ is connected.
Assume that $X(G)\rightarrow X(K\otimes_k G)$ is surjective, where
$K$ is the integral closure of $k$ in $S$.
Let $X^1_G(S)$ be the set of height one $G$-stable prime ideals of $S$.
Let $M(G)$ be the subgroup of the class group $\Cl(S)$ of $S$ generated
by the image of $X^1_G(S)$.
Let $\Gamma$ be a subset of $X^1_G(S)$ whose image in $M(G)$ generates
$M(G)$.
Set $A:=S_G$.
Assume that $Q_G(S)_G\subset Q(A)$.
Assume that if $P\in\Gamma$, then either the height of $P\cap A$ is not one
or $P\cap A$ is principal.
Then for any 
$G$-stable height one prime ideal $Q$ of $S$, either the height of 
$Q\cap A$ is not one or $Q\cap A$ is a principal ideal.
In particular, $A$ is a UFD.
If, moreover, $X(G)$ is trivial, then $S^G=A$ is a UFD.
\end{trivlist}

Here $Q_G(S)$ is the largest $G$-algebra 
contained in the field of fractions $Q(S)$ such that $S$ is a 
$G$-subalgebra of $Q_G(S)$.
$S_G$ (resp.\ $Q_G(S)_G$) is the $k$-subalgebra of $S$ (resp.\ $Q_G(S)$) 
generated by the semiinvariants of $S$ (resp.\ $Q_G(S)$) 
under the action of $G$.
If $X(G)$ is trivial, then $S_G=S^G$ and $Q_G(S)_G=Q(S)^G$.

Note that under the assumption of Theorem~\ref{popov.thm}, 
if $P$ is a height one $G$-stable prime ideal, then
$P=Sf$ for some semiinvariant $f$ of $S$, and $P\cap A=Af$ is principal, 
see Lemma~\ref{principal.thm}, Lemma~\ref{pop-lemma.thm}, 
and Lemma~\ref{principal-invariant.thm}.
Theorem~\ref{main.thm} is a generalization of Theorem~\ref{popov.thm}, 
{\bf (a)}, in the sense that the assumption of Theorem~\ref{popov.thm}, 
{\bf (a)} is stronger than that of Theorem~\ref{main.thm}, and the
conclusion of Theorem~\ref{popov.thm}, {\bf (a)} is weaker than that of
Theorem~\ref{main.thm}.

There are three directions of generalizations.
First, we need not assume that $k$ is algebraically closed.
Second, we need not assume that the 
all $G$-stable height one prime ideals of $S$ are principal.
Third, we treat not only the ring of invariants, but also 
semiinvariants.

We point out that $M(G)$ is the whole class group $\Cl(S)$, if
$k$ is algebraically closed and $G$ is unipotent \cite{FMSS}.
An example such that $M(G)$ is nontrivial but the theorem is applicable and
$S^G$ is a UFD is shown in (\ref{ex-1.par}).
Another generalization is stated as Theorem~\ref{ufd.thm}.
In the theorem, we do not assume that $S$ is a UFD, either
(we do not even assume that $S$ is normal).

We also state and prove some classical results
and their variations on factoriality of (semi-)invariant subrings.
For example, in Theorem~\ref{popov.thm} for the case that {\bf (a)} is
assumed, the finite generation of $S$ is unnecessary, 
the assumption that $k$ is algebraically closed can be weakened,
and we can treat the ring of semiinvariants
(Proposition~\ref{ufd-cor2.thm}).
Theorem~\ref{popov.thm} for the case that {\bf (b)} is assumed is 
also stated as a result on the ring of invariants over $k$ which is
not necessarily algebraically closed 
(Lemma~\ref{ufd-easier.thm}).

Because of the lack of sufficiently many rational points in $G$,
some Hopf algebra technique is required in the discussion over a 
field not necessarily algebraically closed, and our generalization
in this direction is nontrivial.

If $G(k)$ is dense in $G$ (e.g., $k$ is separably closed 
\cite[(AG13.3)]{Borel}, or 
$G$ is connected and $k$ is perfect and infinite \cite{Rosenlicht2}), 
then we can discuss using the action of $G(k)$ on $S$.
This action is extended to that on $Q(S)$, the field of fractions of $S$.
When $G(k)$ is not dense in $G$, then there is some difficulty in treating
$Q(S)$.
We define a substitute $Q_G(S)$ of $Q(S)$, and call it
the {\em $G$-total ring of fractions} of $S$, see (\ref{Q_F(S).par}).
It is the maximal subring of $Q(S)$ on which $G$ acts
such that $S\rightarrow Q_G(S)$ is a $G$-algebra map.
We prove some basic results on $Q_G(S)$.
If $S$ is a Noetherian Nagata domain, then $Q_G(S)$ is a Krull domain
(Corollary~\ref{nagata.thm}).
This machinery plays an important role in proving Theorem~\ref{main.thm}
in the generality of the stated form.

Section~2 is preliminaries.
We give some remarks on divisorial ideals on Krull domains.
We also give some basic results on group actions on rings.

In section~3, we introduce $Q_G(S)$, and study basic properties of it.

In section~4, we give the main results on factoriality of the rings 
generated by the semiinvariants or invariants.

In section~5, we discuss when $Q(S^G)=Q(S)^G$ holds for a
$G$-algebra domain $S$.
This problem was called the {\em Italian problem} in \cite[(6.1)]{Mukai}.
Proposition~\ref{Kamke.thm} 
shows that under the assumption of Theorem~\ref{popov.thm}, 
$Q(S^G)=Q(S)^G$ holds.
We also give a criterion of $Q(S^G)=Q(S)^G$, using the comparison of the
maximum dimension of orbits and $\tdeg_k Q(S)-\tdeg_k Q(S^G)$ 
(Corollary~\ref{Italian2.thm}).
This criterion will be useful in section~6.

In section~6, we give four examples.
The first one is an example of Theorem~\ref{main.thm}.
The second one shows that there is a finitely generated UFD $S$ over
an algebraically closed field $k$ and a finite group $G$ acting on $S$
such that there is no nontrivial group homomorphism $G\rightarrow S^\times$,
but $S^G$ is not a UFD.
So in Theorem~\ref{popov.thm}, the assumption {\bf (iii)} 
cannot be removed.
The third 
example shows that the condition
on the character group imposed in
the statement of Lemma~\ref{principal.thm} is really necessary.
The fourth example shows that the surjectivity of $X(G)\rightarrow
X(K\otimes_k G)$
in Corollary~\ref{ufd-cor.thm} cannot be removed.

It is natural to ask what we can say about \cite[Remark~3, p.~376]{Popov2}
when we consider non-algebraically closed base 
field.
Although we gave some partial results in this paper,
the author does not know the complete answer.
In particular, the author cannot answer the following question.

\begin{question}
Let $k$ be a field (not necessarily algebraically closed), and $G$ an
affine algebraic group over $k$.
Assume that the character group $X(\bar k\otimes_k G)$ is trivial.
Let $S$ be a $G$-algebra UFD with $S^\times\subset S^G$.
Then is $S^G$ a UFD?
\end{question}

\medskip
Acknowledgement: The author is grateful to 
Professor G.~Kemper,
Professor S.~Mukai, 
Dr.\ M.~Ohtani,
and 
Professor V.~L.~Popov
for valuable advice.

\section{Preliminaries}

\paragraph Let $B$ be a commutative ring.
The set of height one prime ideals of $B$ is denoted by $X^1(B)$.

\paragraph
Let $k$ be a field, $S$ a finitely generated $k$-algebra which is a 
normal domain.
Let $k\subset K\subset Q(S)$ be an intermediate field, where $Q(S)$ is
the field of fractions of $S$.
Set $A:=K\cap S$.

\begin{lemma}\label{exceptional.thm}
Let $k$, $S$, $K$, and $A$ be as above.
There are only finitely many height one prime ideals $P$ of $S$
such that $\height(P\cap A)\geq 2$.
\end{lemma}

\begin{proof} 
Replacing $K$ by $Q(A)$ if necessary, we may assume that $K=Q(A)$.
There is a finitely generated normal $k$-subalgebra $B$ of $A$ such that
$Q(B)=Q(A)=K$.
Let $W:=\{P\in\Spec S\mid S_P \text{ is not flat over $B$}\}$,
and $Z$ the image of $W$ by the canonical map $\Spec S\rightarrow\Spec B$.
Then $W$ is closed in $\Spec S$, and $Z$ is constructible in $\Spec B$ 
\cite[Theorem~6]{Matsumura}.
As a vector space over a field is flat and a torsion free module over a
DVR is flat, the closure $\bar Z$ of $Z$ has codimension at least two 
in $\Spec B$.
There are only finitely many subvarieties of codimension one in $\Spec S$
which are mapped to $\bar Z$.
So it suffices to prove that if $P$ is a height one prime ideal of $S$
such that $V(P\cap B)\not\subset\bar Z$, then $\height(P\cap A)\leq 1$.
Consider the local homomorphism
$
B_{P\cap B}\hookrightarrow A_{P\cap A}\subset K.
$
If $\height(P\cap B)=0$, then $B_{P\cap B}=K$ and this forces
$A_{P\cap A}=K$, and $\height(P\cap A)=0$.
If $\height(P\cap B)=1$, then $B_{P\cap B}$ is a DVR with $Q(B_{P\cap B})=K$.
Since the map $B_{P\cap B}\hookrightarrow A_{P\cap A}$ is local, 
$B_{P\cap B}=A_{P\cap A}$, and hence $\height(P\cap A)=1$.
$\height(P\cap B)\geq 2$ cannot happen by flatness.
\end{proof}

\paragraph
Let $k$, $S$, $K$, and $A$ be as above.
For $\fp\in X^1(A)$, we define
\[
X^1(\fp)=\{P\in X^1(S)\mid P\cap A=\fp\}.
\]
Note that $X^1(\fp)$ is finite.
For $P\in X^1(\fp)$, define $m(P)$ by $\fp S_P=P^{m(P)}S_P$.
Let $v_P$ be the normalized discrete valuation associated with $P\in X^1(S)$.
For a subset $I$ of $S$, 
define $v_P(I)=\inf\{v_P(x)\mid x\in I\}$.
Note that $IS_P=(PS_P)^{v_P(I)}$, where $(PS_P)^{-\infty}=Q(S)$, and
$(PS_P)^{\infty}=0$.
Thus $m(P)=v_P(\fp)$ for $\fp\in X^1(A)$ and $P\in X^1(\fp)$.
An $S$-submodule $I$ of $Q(S)$ is a divisorial fractional ideal 
if and only if $v_P(I)\in\Bbb Z$ for any $P\in X^1(S)$, $v_P(I)\neq 0$
for only finitely many $P$, and 
$I=\bigcap_{P\in X^1(S)} (PS_P)^{v_P(I)}$ (in case $S$ is a field,
the right hand side should be understood to be $S=Q(S)$.
We use this convention in the sequel).

\paragraph\label{class-group.par}
For a Krull domain $B$, the class group of $B$, denoted by 
$\Cl(B)$, is the $\Bbb Z$-free module 
$\Div(B):=\bigoplus_{P\in X^1(B)}\Bbb Z\cdot \langle P\rangle$
with the free basis $X^1(B)$ (the basis element corresponding to $P
\in X^1(B)$ is denoted by $\langle P\rangle$), modulo the subgroup
$\Prin(B):=\{\divisor(a)\mid a\in Q(B)\setminus\{0\}\}$,
where $\divisor a=\sum_{P\in X^1(B)}v_P(a)\langle P\rangle$.
That is, $\Cl(B):=\Div(B)/\Prin(B)$.
The class of $D\in \Div(B)$ in $\Cl(B)$ is denoted by $\bar D$.
For $D=\sum_P c_P \langle P\rangle \in \Div(B)$, we can associate
a divisorial fractional ideal
\[
\I(D):=\bigcap_{P}(PA_P)^{-c_P}=\{a\in Q(B)\setminus\{0\}
\mid D+\divisor a\geq 0\}\cup\{0\},
\]
where for $D'=\sum_P c_P'\langle P\rangle\in \Div (B)$, 
we say that $D'\geq 0$ if $c_P'\geq 0$ for all $P\in X^1(B)$.
The map $D\mapsto \I(D)$ induces an isomorphism between $\Div(B)$ and
the group of divisorial fractional ideals $\DF(B)$ of $B$,
where for $I,J\in \DF(B)$, the sum of $I$ and $J$ in $\DF(B)$ 
is defined to be $B:_{Q(B)}(B:_{Q(B)}IJ)$.
Note that $\I$ induces an isomorphism $\bar \I$ between $\Cl(B)$ and the
group of the isomorphism classes of the divisorial fractional ideals 
$\Cl'(B)$ of $B$.
We identify $\Cl(B)$ and $\Cl'(B)$ via $\bar \I$.
For a divisorial fractional ideal $I$ of $B$, the class of $I$ in 
$\Cl'(B)$ is denoted by $[I]$.
If $I=\bigcap_P (PA_P)^{c_P}$, then $[I]=\sum_P c_P[P]$.
Note that $\bar \I(\overline{\langle P\rangle})=-[P]$.

\begin{lemma}\label{divisorial.thm}
Let $I$ be a divisorial fractional ideal of $S$.
Assume that $I\cap Q(A)$ is a divisorial fractional ideal of $A$.
\begin{description}
\item[(i)] If $P\in X^1(S)$ and $P\cap A=0$, then $v_P(I)\leq 0$.
\item[(ii)] $I\cap Q(A)=\bigcap_{\fp\in X^1(A)}(\fp A_{\fp})^{n_\fp}$,
where $n_\fp=n_\fp(I)=\max\{\ru{v_P(I)/m(P)}\mid P\in X^1(\fp)\}$,
where $\ru{?}$ is the ceiling function.
\end{description}
\end{lemma}

\begin{proof} 
{\bf (i)} If $P\cap A=0$, then $v_P(x)=0$ for $x\in A\setminus\{0\}$.
This shows that $v_P(\alpha)=0$ for $\alpha\in Q(A)\setminus\{0\}$.
So if, moreover, $v_P(I)>0$, then $I\cap Q(A)=0$.
But as we assume that $I\cap Q(A)$ is a fractional ideal 
(in particular, nonzero), this is absurd.

{\bf (ii)} 
Note that
\[
I\cap Q(A)=\bigcap_{P\in X^1(S)} ((PS_P)^{v_P(I)}\cap Q(A)).
\]
For $P\in X^1(S)$ such that $P\cap A=0$, 
$v_P(I)\leq 0$ by {\bf (i)}, and hence $(PS_P)^{v_P(I)}\cap Q(A)=Q(A)$,
and such a $P$ can be removable from the intersection.
Set 
\[
J_1=J_1(I)=
\bigcap_{\fp\in X^1(A)}\bigcap_{P\in X^1(\fp)}((PS_P)^{v_P(I)}\cap Q(A))
\]
and
\begin{equation}\label{J_2.eq}
J_2=
\bigcap_{P\in X^1(S),\;\height(P\cap A)\geq 2}((PS_P)^{v_P(I)}\cap Q(A)).
\end{equation}
Then $I\cap Q(A)=J_1\cap J_2$.
For $\fp\in X^1(A)$ and $P\in X^1(\fp)$, $(PS_P)^{v_P(I)}\cap Q(A)$ is
an $A_\fp$-submodule of $Q(A)$.
On the other hand, for any $A_\fp$-submodule $M$ of $Q(A)$, 
$M= (\fp A_\fp)^{v_P(M)/m(P)}$.
As 
\[
v_P((PS_P)^{v_P(I)}\cap Q(A))\geq m(P)\cdot\ru{v_P(I)/m(P)}, 
\]
$(PS_P)^{v_P(I)}\cap Q(A)\subset (\fp A_\fp)^{\ru{v_P(I)/m(P)}}$.
The opposite inclusion is trivial, and hence
\[
(PS_P)^{v_P(I)}\cap Q(A)=(\fp A_\fp)^{\ru{v_P(I)/m(P)}}.
\]
This immediately leads to
\begin{equation}\label{J_1.eq}
J_1=\bigcap_{\fp\in X^1(A)}(\fp A_{\fp})^{n_\fp}.
\end{equation}
As $v_P(I)\in\Bbb Z$ for $P\in X^1(S)$, and $v_P(I)\neq 0$ for 
only finitely many $P$, it follows that
$n_\fp\in\Bbb Z$ for $\fp\in X^1(A)$, and $n_\fp\neq 0$ for
only finitely many $\fp$.
In particular, $J_1$ is a divisorial fractional ideal, and
$v_\fp(J_1)=n_\fp$ for $\fp\in X^1(A)$.

It suffices to show that $I\cap Q(A)=J_1$.
As the both hand sides are divisorial fractional ideals, 
it suffices to show that
$(I\cap Q(A))_\fp=(J_1)_\fp$ for any $\fp\in X^1(A)$.
As $I\cap Q(A)=J_1\cap J_2$, this is equivalent to say that
$(J_2)_\fp\supset (J_1)_\fp$ for any $\fp$.

Now first consider the case that $J_1\subset A$.
That is to say, $n_\fp\geq 0$ for any $\fp\in X^1(A)$.
Then it suffices to show that $(J_2)_\fp\supset 
A_\fp$, since $(J_1)_\fp\subset A_\fp$.
As the intersection in (\ref{J_2.eq}) is finite by Lemma~\ref{exceptional.thm},
It suffices to show that
$((PS_P)^{v_P(I)}\cap Q(A))_\fp\supset A_\fp$ for any $P\in X^1(S)$ such that
$\height(P\cap A)\geq 2$, and any $\fp\in X^1(A)$.
But this is trivial, since $P^n\cap A\supset (P\cap A)^n\not\subset \fp$
for $n\geq 0$.
Now the lemma is true for $I$ such that $J_1\subset A$.

Next consider the general case.
Take $a\in A\setminus \{0\}$ such that $aJ_1\subset A$.
Let $\divisor a=\sum_\fp c(\fp)\langle\fp\rangle$.
Note that $aI\cap Q(A)=a(I\cap Q(A))$ is divisorial.
It is easy to see that $n_\fp(aI)=n_\fp(I)+c(\fp)$.
Or equivalently, $J_1(aI)=a J_1(I)$.
So the lemma is true for $aI$.
That is, $aI\cap Q(A)=J_1(aI)=a J_1(I)$.
So $I\cap Q(A)=a^{-1}(aI\cap Q(A))=a^{-1}a J_1(I)=J_1(I)$,
and the lemma is also true for this general $I$.
\end{proof}

\begin{lemma}\label{submod.thm}
Let $R$ be a commutative ring, $C$ an $R$-coalgebra, and 
$A\rightarrow C^*$
be a universally dense $R$-algebra map.
That is, an $R$-algebra map such that for any $R$-module $M$, 
$\theta_M: M\otimes C\rightarrow \Hom_R(A,M)$
\($(\theta_M(m\otimes c))(a)=(ac)\cdot m$\) is injective.
A \(right\) $C$-comodule is a \(left\) $C^*$-module, and it is an $A$-module.
Let $V$ be a $C$-comodule and $W$ its $R$-submodule.
Then $W$ is a $C$-subcomodule of $V$ if and only if it is an $A$-submodule.
\end{lemma}

\begin{proof} 
Note that $C$ is $R$-flat \cite[(I.3.8.4)]{Hashimoto}.
Consider the map $\rho: W\rightarrow V/W\otimes C$ defined by 
$\rho=(\pi\otimes 1_C)\circ \omega_V\circ \iota$, where $\iota:W\rightarrow 
V$ is the inclusion, and $\pi:V\rightarrow V/W$ is the projection.
$W$ is a $C$-subcomodule of $V$ if and only if $\rho$ is zero.
On the other hand, $W$ is an $A$-submodule of $V$ if and only if
$\theta_{V/W}\circ \rho:W\rightarrow \Hom_R(A,V/W)$ is zero,
since $((\theta_{V/W}\circ \rho)(w))(a)=\overline{aw}$.
Since $\theta_{V/W}$ is injective, $\rho=0$ if and only if
$\theta_{V/W}\circ \rho=0$.
The assertion follows.
\end{proof}

\begin{lemma}\label{subcomod.thm}
Let $R$ be a commutative ring, and 
$C$ an $R$-flat coalgebra.
Let $V$ be a $C$-comodule, and $W$ its $R$-submodule.
Let $R'$ be an $R$-algebra.
Let $M'=R'\otimes_R M$ for an $R$-module $M$.
If $W$ is a $C$-subcomodule of $V$, then $W'$ is a $C'$-subcomodule of $V'$.
If $R'$ is faithfully flat over $R$ and $W'$ is a $C'$-subcomodule of $V'$,
then $W$ is a $C$-subcomodule of $V$.
\end{lemma}

\begin{proof} 
If $\rho: W\rightarrow V/W\otimes_R C$ in the proof of
Lemma~\ref{submod.thm} is zero, then so is the base change
$\rho': W'\rightarrow V'/W'\otimes_{R'} C'$.
The converse is true, if $R'$ is faithfully flat over $R$.
\end{proof}

\paragraph Let $R$ be a commutative ring.
A sequence of $R$-group schemes
\[
X\xrightarrow{\varphi}Y\xrightarrow{\psi}Z
\]
is said to be exact, if $\varphi(X)=\Ker\psi$ as $R$-subfaisceaux of $Y$, 
see \cite[(I.5.5)]{Jantzen}.

\begin{lemma}\label{exact.thm}
Let $R$ be a Noetherian commutative ring, and
\begin{equation}\label{exact3.eq}
1\rightarrow N\xrightarrow{\varphi} F\xrightarrow{\psi} T\rightarrow 1
\end{equation}
an exact sequence of affine flat $R$-group schemes.
Then 
\begin{description}
\item[(i)] $\varphi$ is a closed immersion.
\item[(ii)] $\psi$ is faithfully flat, and $R[T]\rightarrow R[F]$ is injective.
\item[(iii)] For any $R$-module $V$, $V\otimes R[T]\rightarrow V\otimes R[F]$
induces $V\otimes R[T]\cong (V\otimes R[F])^N$, where $N$ acts on $R[F]$
right regularly.
\item[(iv)] For any $F$-module $M$, $M^N$ is uniquely a $T$-module so that
$M^N$ as an $F$-module is an $F$-submodule of $M$.
\end{description}
\end{lemma}

\begin{proof}
{\bf (i)} This is because $N=\Ker \psi$ is the equalizer of the two morphisms
$\psi$ and the trivial map, and $T$ is $R$-separated.

{\bf (ii)} This is \cite[(I.5.7)]{Jantzen}.

{\bf (iii)}
Consider the augmented cobar complex
\begin{equation}\label{exact2.eq}
0\rightarrow R[T]\xrightarrow{\psi'} 
R[F]\rightarrow R[F]\otimes_R R[N]\rightarrow
\cdots
\end{equation}
For any algebraically closed field $K$ which is an $R$-algebra,
(\ref{exact3.eq}) base changed to $K$ is a short exact sequence of
affine algebraic $K$-group schemes.
As $(K\otimes_R F)/(K\otimes_R N)\cong K\otimes_R T$ is affine, 
$H^i(K\otimes_R N,K\otimes_R R[F])=0$ for $i>0$, and
$H^0(K\otimes_R N,K\otimes_R R[F])\cong K\otimes_R R[T]$.
It follows that (\ref{exact2.eq}) base changed to $K$ is exact.
By \cite[Corollary~3]{Hashimoto2}, for any $R$-module $V$, 
$V\otimes R[T]\cong (V\otimes R[F])^N$.

{\bf (iv)}
Via the injective map $\psi' :R[T]\rightarrow R[F]$, we identify
$R[T]$ with a subcoalgebra of $R[F]$ (for the definition of a subcoalgebra, 
see \cite[(I.3.6.7)]{Hashimoto}).
Let $\omega:M\rightarrow M\otimes R[F]$ be the coaction.
We are to prove that $\omega(M^N)\subset M^N\otimes_R R[T]$.

First we prove that $\omega(M^N)\subset M\otimes R[T]$.
By {\bf (iii)}, this is equivalent to say that the composite map
\[
M^N\hookrightarrow M\xrightarrow{\omega}M\otimes R[F]
\xrightarrow{1\otimes\Delta}
M\otimes R[F]\otimes R[F]\xrightarrow{1\otimes \varphi'}M\otimes R[F]\otimes
R[N]
\]
maps $m\in M^N$ to $\omega(m)\otimes1$, 
where $\Delta:R[F]\rightarrow R[F]\otimes R[F]$ is the coproduct.
By the coassociativity, this map agrees with
\[
M^N\hookrightarrow M\xrightarrow{\omega}M\otimes R[F]
\xrightarrow{1\otimes\varphi'}M\otimes R[N]
\xrightarrow{\omega\otimes 1}M\otimes R[F]\otimes R[N],
\]
and this map sends $m$ to $\omega(m)\otimes 1$ by the definition of $M^N$.

Next we prove that $\omega(M^N)\subset M^N\otimes R[F]$.
As $N$ is a normal subgroup of $F$, 
the map $\rho:R[F]\rightarrow R[N]\otimes R[F]$ given by
$f\mapsto\sum_{(f)}\bar f_2\otimes (Sf_1)f_3$ factors through 
$R[N]$,
where $S$ denotes the 
antipode of $R[F]$, and we are employing Sweedler's notation, see 
\cite[(I.3.4)]{Hashimoto}, for example.
In order to prove $\omega(M^N)\subset M^N\otimes R[F]$, it suffices to show
that $\sum_{(m)}m_0\otimes \bar m_1\otimes m_2=\sum_{(m)}m_0\otimes 1
\otimes m_1$ for $m\in M^N$.

Note that
\[
\sum_{(m)}m_0\otimes \bar m_1\otimes m_2
=\sum_{(m)}m_0\otimes \bar m_3\otimes m_1(Sm_2)m_4
\]
is the image of $\sum_{(m)}m_0\otimes\rho(m_1)$ by the map
\[
\gamma:M\otimes R[N]\otimes R[F]\rightarrow
M\otimes R[N]\otimes R[F]
\qquad
(m\otimes \bar f\otimes f'\mapsto \sum_{(m)}m_0\otimes \bar f\otimes m_1f').
\]
Since $\rho$ factors through $R[N]$ and decompose like
$\rho=\bar \rho \varphi'$ and $\sum_{(m)}m_0\otimes \bar m_1=m\otimes 1$, 
\begin{multline*}
\sum_{(m)}m_0\otimes \bar m_1\otimes m_2
=
\gamma(\sum_{(m)}m_0\otimes \bar\rho(\bar m_1))
=\gamma(m\otimes\rho(1))
\\
=\gamma(m\otimes 1\otimes 1)=\sum_{(m)}m_0
\otimes 1\otimes m_1,
\end{multline*}
as desired.

To conclude the proof, it suffices to prove that
$M^N\otimes R[F]\cap M\otimes R[T]=M^N\otimes R[T]$, as submodules of
$M\otimes R[F]$.
This follows from a straightforward diagram chasing of the
commutative diagram with exact rows and columns
{\scriptsize
\[
\xymatrix{
 & 0 \ar[d] & 0 \ar[d] & 0 \ar[d] \\
0 \ar[r] & M^N\otimes R[T] \ar[r] \ar[d] &
M\otimes R[T] \ar[r] \ar[d] &
M\otimes R[N]\otimes R[T] \ar[d] \\
0 \ar[r] & M^N\otimes R[F] \ar[r] \ar[d] &
M\otimes R[F] \ar[r] \ar[d] &
M\otimes R[N]\otimes R[F] \ar[d] \\
0 \ar[r] & M^N\otimes R[F]\otimes R[N] \ar[r] &
M\otimes R[F]\otimes R[N] \ar[r] &
M\otimes R[N]\otimes R[F]\otimes R[N].
}
\]
}
The rows are exact, since
\[
0\rightarrow M^N\rightarrow M\xrightarrow{\omega-\iota}M\otimes R[N]
\]
is exact and $R[T]$, $R[F]$, and $R[N]$ are $R$-flat, where
$\iota(m)=m\otimes 1$.
The columns are exact by {\bf (iii)}.
This completes the proof of {\bf (iv)}.
\end{proof}

\paragraph
Until the end of this paper,
let $k$ be a field, and $G$ an affine algebraic $k$-group.
That is, an affine algebraic $k$-group scheme that is smooth over $k$.
Let $H$ be an affine algebraic $k$-group scheme.
For a $k$-algebra $A$, let us denote $H(A)$ the group of $A$-valued 
points of $H$.
However, if $G$ is a finite (constant) 
group (over $k$), then the group $G(k)$ is sometimes
simply denoted by $G$ by an obvious reason.
Let $\MM H$ denote the category of $H$-modules.
The coordinate ring of $H$ is denoted by $k[H]$.
For an abstract group $\Gamma$, let $\MM \Gamma$ denote the
category of $k\Gamma$-modules, where $k\Gamma$ is the group ring of 
$\Gamma$ over $k$.
Let $\bar k$ and $k\sep$ respectively denote the algebraic and the
separable closure of $k$.
The characteristic of $k$ is denoted by $\charac(k)$.

\begin{lemma}\label{dense.thm}
Assume that $G(k)$ is dense in $G$ with respect to the Zariski topology.
Then
\begin{description}
\item[(i)] The canonical functor $\MM G\rightarrow \MM {G(k)}$ is
full and faithful.
\item[(ii)]
For a $G$-module $V$ and its subspace $W$, $W$ is a $G$-submodule of $V$
if and only if it is a $G(k)$-submodule.
\item[(iii)]
For a $G$-module $V$, we have $V^G=V^{G(k)}$.
\end{description}
\end{lemma}

\begin{proof} 
By assumption, $k[G]\rightarrow (k G(k))^*$ ($f\mapsto (g\mapsto 
f(g))$ for $f\in k[G]$ and $g\in G(k)$) is injective.
By \cite[(I.3.9.1)]{Hashimoto}, $k G(k)\rightarrow k[G]^*$ is 
universally dense.
By \cite[(I.3.10.3)]{Hashimoto}, the first assertion {\bf (i)} follows.
The assertion {\bf (ii)} is a consequence of Lemma~\ref{submod.thm}.
We prove {\bf(iii)}.
\[
V^G=\sum_{\phi\in\Hom_G(k,V)}\Image\phi=
\sum_{\phi\in\Hom_{kG(k)}(k,V)}\Image\phi=V^{G(k)}
\]
by {\bf (i)}.
\end{proof}

\begin{lemma}\label{connected-irreducible.thm}
Let $H$ be an affine algebraic $k$-group scheme.
Then the identity component \(that is, the connected component of $H$
containing the image of the unit element $\Spec k\rightarrow H$\) 
$H^\circ$ is a normal subgroup scheme of $H$.
$H^\circ$ is geometrically irreducible.
In particular, if $H$ is connected, then $H$ is geometrically irreducible.
\end{lemma}

\begin{proof}
The case that $H$ is smooth is \cite[(I.1.2)]{Borel}.

Next we consider the case that $k$ is perfect.
Set $E=H\red$.
Then $E$ is a closed subgroup scheme of $H$, and is smooth.
Clearly, $E\hookrightarrow H$ is a homeomorphism, and $E^\circ$ is
identified with $(H^\circ)\red$.
The product $E^\circ\times E^\circ\rightarrow E$ factors through $E^\circ$,
and the inverse $E^\circ\rightarrow E$ factors through $E^\circ$.
Moreover, $E\times E^\circ\rightarrow E$ ($(e,n)\mapsto ene^{-1}$) factors
through $E^\circ$ by normality.
As $H^\circ$ is an
open subscheme of $H$ and $E^\circ$ is set theoretically the same as $H^\circ$,
$H^\circ\times H^\circ\rightarrow H$ factors through $H^\circ$, 
and the inverse $H^\circ\rightarrow H$ factors through $H^\circ$.
Moreover, $H\times H^\circ\rightarrow H$ ($(h,m)\mapsto hmh^{-1}$)
factors through $H^\circ$.
Namely, $H^\circ$ is a normal subgroup scheme of $H$.
Finally, as $E^\circ$ is geometrically irreducible, $H^\circ$ is so.

Next consider the general case.
Let $K$ be the smallest perfect field containing $k$.
Then the canonical map $\rho:K\otimes_k H\rightarrow H$ is a homeomorphism.
So $(K\otimes_k H)^\circ=\rho^{-1}(H^\circ)=K\otimes_k H^\circ$.
As it is a normal $K$-subgroup scheme of $K\otimes_k H$ and is geometrically
irreducible, $H^\circ$ is a normal $k$-subgroup scheme of $H$ and is
geometrically irreducible.

The last assertion is trivial.
\end{proof}

\paragraph Let $X_0$ be a scheme, and $F$ a flat $X_0$-group scheme.
Let $X$ be an $F$-scheme.
For an ideal (quasi-coherent or not) $\I$ of $\O_X$, the sum of all the
$F$-stable quasi-coherent ideals of $\I$ is the largest $F$-stable 
quasi-coherent ideal of $\O_X$ contained in $\I$.
We denote this by $\I^*$ as in \cite[section~4]{HM}.
If $Y$ is a closed subscheme of $X$ and $Y=V(\I)$, then we denote
$V(\I^*)=Y^*$.
Note that $Y^*$ is the smallest $F$-stable closed subscheme of $X$ containing
$Y$.
If $S$ is an $F$-algebra and $I$ is an ideal of $S$, there is a unique
largest $F$-stable ideal $I^*$ contained in $I$.

\begin{lemma}\label{star-integral.thm}
Let $X_0$ be a scheme, $F$ a
flat
quasi-separated $X_0$-group scheme of finite type 
with connected \(resp.\ smooth, smooth and connected\) 
fibers.
Let $X$ be an $F$-scheme, and $Y$ a closed subscheme of $X$.
Then the scheme theoretic image of the action $a_Y:F\times Y\rightarrow X$
agrees with $Y^*$.
If, moreover, 
$Y$ is irreducible \(resp.\ reduced, integral\), 
then $Y^*$ is so.
\end{lemma}

\begin{proof}
Let $Y'$ denote the scheme theoretic image of $a_Y$.
If $Z$ is a closed $F$-stable subscheme of $X$ containing $Y$, then
$a_Y$ is the composite
\[
F\times Y \hookrightarrow F\times Z\xrightarrow{a_Z} Z\hookrightarrow X,
\]
and factors through $Z$.
This shows $Y'\subset Z$ by the definition of $Y'$.
On the other hand, $Y'\supset Y$ is trivial, and $Y'$ is $F$-stable,
since $a_Y$ is $F$-stable (where $F$ acts on $F\times Y$ by 
$g\cdot(g',y)=(gg',y)$) and quasi-compact quasi-separated, 
and $Y'$ is defined by the $F$-stable quasi-coherent ideal
$\Ker(\O_X\rightarrow (a_Y)_*(\O_{F\times Y}))$.
Thus $Y'$ is the smallest closed $F$-stable subscheme of $X$ 
containing $Y$, and agrees with $Y^*$.

If $F$ has connected fibers and $Y$ is irreducible, 
then $F\times Y$ is irreducible, since $F$ is geometrically irreducible
over $X_0$ by Lemma~\ref{connected-irreducible.thm}.
If $F$ is smooth and $Y$ is reduced, then $F\times Y$ is reduced.
So if $F$ is smooth with connected fibers and $Y$ is integral,
then $F\times Y$ is integral.

Let $\alpha: F\times Y\rightarrow Y^*$ be the morphism induced by $a_Y$.
As $F\times Y$ is irreducible (resp.\ reduced, integral) and
$\O_{Y^*}\rightarrow \alpha_*\O_{F\times Y}$ is injective, 
$Y^*$ is irreducible (resp.\ reduced, integral).
\end{proof}

\begin{lemma}\label{idempotent.thm}
Let $H$ be a connected affine $k$-group scheme, 
and $S$ an $H$-algebra.
If $e$ is an idempotent of $S$, then $e\in S^H$.
\end{lemma}

\begin{proof}
Let $V$ be a finite dimensional $H$-submodule of $S$
containing $e$, and $S_1$ the $k$-subalgebra of $S$ generated by $V$.
Then $S_1$ is of finite type and $e\in S_1$.
So replacing $S$ by $S_1$, we may assume that $S$ is of finite type.

Set $X=\Spec S$, $X_1=\Spec Se$ and $X_2=\Spec S(1-e)$.
Both $X_1$ and $X_2$ are closed open subsets of $X$, and $X=X_1\coprod X_2$.
For any irreducible component $Y$ of $X_2$, $Y^*$ is irreducible
by Lemma~\ref{star-integral.thm}.
As $Y$ is an irreducible component of $X$ and $Y\subset Y^*$, we have
that $Y^*=Y$, set theoretically.
So $H\times X_2\rightarrow X$ factors through $X_2$, set theoretically.
As $X_2$ is an open subscheme, this shows that $X_2$ is $H$-stable.
That is, the action $H\times X_2\rightarrow X$ factors through $X_2$.
On the other hand, $X_2$ is a closed subscheme of $X$ defined by the
ideal $Se$.
So $Se$ is an $H$-stable ideal.
That is, the coaction $\omega:S\rightarrow S\otimes k[H]$ maps $Se$ to
$Se\otimes k[H]$.
Similarly, $\omega(1-e)\in S(1-e)\otimes k[H]$.

So
\[
\omega(e)=(e\otimes 1)\omega(e)=
(e\otimes 1)\omega(1-(1-e))
=e\otimes 1 - (e\otimes 1)(\omega(1-e))=e\otimes 1,
\]
as desired.
\end{proof}

\paragraph\label{separable.par}
Let $A$ be a $k$-algebra.
We say that $A$ is {\em geometrically reduced} over $k$ if 
$K\otimes_k A$ is reduced for any finite
extension field $K$ of $k$.
This is equivalent to say that $A$ is reduced, if $\charac(k)=0$.
If $\charac(k)=p>0$,
this is equivalent to say that $k^{-p}\otimes_k A$ is reduced.
When $A$ is a field, $A$ is separable over $k$ if and only if
it is geometrically reduced over $k$.

Clearly, $A$ is geometrically reduced if and only if any finitely generated
$k$-subalgebra of $A$ is geometrically reduced.
Any localization of a geometrically reduced algebra is again 
geometrically reduced.
In particular, $A$ is geometrically reduced over $k$ if and only if $Q(A)$ is
geometrically reduced over $k$.

Let $A$ be a reduced $k$-algebra which is integral over $k$.
We say that $a\in A$ is separable over $k$ if $f(a)=0$ for some
monic polynomial $f(x)\in k[x]$ without multiple roots.
The set of separable elements of $A$ is the largest geometrically reduced
$k$-subalgebra of $A$.

\begin{lemma}\label{connected-trivial.thm}
Let $K$ be a $G$-algebra which is a finite direct product of 
finite algebraic extension fields of $k$.
If $G$ is connected, then the action of $G$ on $K$ is trivial.
\end{lemma}

\begin{proof} 
Taking the base change and replacing $k$ by its separable closure,
we may assume that $k$ is separably closed.
It is obvious that the set of separable elements $L$ of $K$ over $k$ is a 
$G(k)$-subalgebra (i.e., a $k$-subalgebra which is also a $G(k)$-submodule) 
of $K$.
As $k$ is separably closed, $G(k)$ is dense in $G$
\cite[(AG13.3)]{Borel}, and
$L$ is a $G$-subalgebra of $K$ by (\ref{dense.thm}).

As $L\cong k^n$ as a $k$-algebra for some $n$, 
$L$ is spanned by its idempotents as a $k$-vector space.
As $G$ is connected, $L$ is $G$-trivial by Lemma~\ref{idempotent.thm}.

Next, we prove that $K$ is $G$-trivial.
Of course we may assume that $k$ is of positive characteristic, say $p$.
For any $\alpha\in K$, there exists some $m>0$ such that 
$\alpha^{p^m}\in L$.
Then for any $g\in G(k)$, $(g\alpha-\alpha)^{p^m}=0$.
As $K$ is a reduced ring, $\alpha\in K^{G(k)}=K^G$.
Hence $K$ is $G$-trivial.
\end{proof}

\paragraph Let $B$ be a $\Bbb Z^n$-graded Krull domain.
Let $X^1\gr(B)$ be the set of height one graded prime ideals.
Let $\Div\gr(B)$ be the free abelian group with the free basis $X^1\gr(B)$.
Let $\Prin\gr(B)$ be the subgroup 
$$\{\divisor a\mid \text{$a$ is nonzero 
homogeneous}\}.$$
Note that $\Prin\gr(B)\subset \Div\gr(B)$.
Indeed, if $a\in B\setminus\{0\}$ is homogeneous and $P$ is a minimal
prime of $a$, then $P$ is of height one.
Unless $P$ is homogeneous, $P\supsetneq P^*$, and $P^*$ is a prime
ideal by Lemma~\ref{star-integral.thm} applied to the action of the
split torus $\Bbb G_m^n$.
This shows $P^*=0$ and contradicts $a\in P^*$. 

\begin{lemma}[cf.~{\cite[Theorem~1.1]{KK}}, 
{\cite[Proposition~10.2]{Fossum}}]\label{KK.thm}
The canonical map $\theta:
\Cl\gr(B):=\Div\gr(B)/\Prin\gr(B)\rightarrow \Cl(B)$ is bijective.
\end{lemma}

\begin{proof}
First we show that $\theta$ is injective.
It suffices to show that for $f\in B\setminus\{0\}$, if $\divisor f\in
\Div\gr(B)$, then $f$ is homogeneous.
Let $\Gamma$ be the multiplicatively closed subset of $B$ consisting of
all the nonzero homogeneous elements of $B$.
Note that the canonical map $\Div(B)\rightarrow\Div(B_\Gamma)$
kills all the homogeneous height one prime ideals, while for any inhomogeneous
height one prime $P$, $\langle P\rangle$ goes to the basis element
$\langle PB_\Gamma\rangle$.
So in $\Div(B_\Gamma)$, $\divisor f =0$.
So $f\in B_\Gamma^\times$.
On the other hand, a unit of a $\Bbb Z^n$-graded domain is homogeneous, and
hence $f$ is homogeneous.

We prove that $\theta$ is surjective.
Let $P\in X^1(B)$.
Since $B_\Gamma$ is a Laurent polynomial algebra over a field, $B_\Gamma$
is a UFD.
So $PB_\Gamma=B_\Gamma f$ for some prime element $f$ of $B_\Gamma$.
This shows that $\langle P\rangle -\divisor f\in\Div\gr(B)$.
Hence $\overline{\langle P\rangle}\in\Image \theta$.
So $\theta$ is surjective.
\end{proof}

\begin{corollary}\label{graded-ufd.thm}
Let $B$ be as above.
If any nonzero homogeneous element is either a unit or divisible by a
prime element, then $B$ is a UFD.
\end{corollary}

\begin{proof}
Let $w:\Div(B)\rightarrow \Bbb Z$ be the map given by $w(\sum_Pc_P\langle
P\rangle)=\sum_P c_P$.

Using induction on $w(\divisor b)$, we prove that any nonzero homogeneous
element $b$ of $B$ is either a unit or has a prime factorization.
If $w(\divisor b)=0$, then $b$ is a unit.
If $w(\divisor b)>0$, then $b$ is not a unit, and $b=pb'$ for some 
prime element $p\in B$ and $b'\in B$.
Note that both $p$ and $b'$ are homogeneous, since $b$ is homogeneous.
As $w(b')<w(b)$, $b'$ is either a unit or a product of prime elements,
and we are done.

Now let $P$ be any homogeneous height one prime ideal.
Then take a nonzero homogeneous element $b\in P$.
As $b$ is a product of prime elements, there exists some prime element
$p$ which lies in $P$.
As $P$ is height one, $P=Bp$ is principal.
As any homogeneous height one prime ideal is principal, 
$B$ is a UFD by Lemma~\ref{KK.thm}.
\end{proof}

\begin{lemma}\label{graded-ufd2.thm}
Let $B$ be a $\Bbb Z^n$-graded domain.
If any nonzero homogeneous element of $B$ is either a unit or a product of
prime elements, then $B$ is a UFD.
\end{lemma}

\begin{proof} 
In view of Corollary~\ref{graded-ufd.thm}, it suffices to show that
$B$ is a Krull domain.

Let $\Gamma$ be the set of nonzero homogeneous elements of $B$.
Then $B_\Gamma$ is a Laurent polynomial ring over a field.
As $B_\Gamma$ is a Noetherian normal domain, $B_\Gamma=\bigcap_P (B_\Gamma)_P$,
where $P$ runs through the set of height one prime ideals of $B_\Gamma$.

We prove that $B=\bigcap_P(B_\Gamma)_P\cap \bigcap B_{\pi B}$, where
$\pi B$ runs through the principal prime ideals of $B$ generated by
homogeneous prime elements.
Let $b/s$ be in the right hand side, where $b\in B$ and $s\in \Gamma$.
We may assume that for each $\pi$, if $\pi$ divides $b$, then $\pi$ does not
divide $s$.
Then $s$ is a unit of $B$, and $b/s\in B$, as desired.

An element $b$ of $B$ lies in only finitely many $P(B_\Gamma)_P$.
On the other hand, if $b$ lies in $\pi B$, then each homogeneous component
of $b$ lies in $\pi B$.
This shows that $b$ lies in only finitely many $\pi B_{\pi B}$.

As $(B_\Gamma)_P$ is obviously a DVR, it remains to show that $B_{\pi B}$ is
a DVR.
To verify this, $\bigcap_{n\geq 0}\pi^n B_{\pi B}=0$ is enough.
So $b/c\in B_{\pi B}\setminus\{0\}$ with $b\in B\setminus\{0\}$ and 
$c\in B\setminus \pi B$.
Note that $b\in \pi^n B$ if and only if each homogeneous component of $b$
lies in $\pi^n B$.
So $b\in \pi^nB\setminus \pi^{n+1}B$ for some $n\geq 0$.
Then it is easy to see that $b/c\in \pi^n B_{\pi B}\setminus \pi^{n+1}
B_{\pi B}$.
Hence $B$ is a Krull domain.
\end{proof}

\section{Equivariant total ring of fractions}

\paragraph 
For a ring $B$, let us denote the set of nonzerodivisors of $B$ by $B\nz$.
The localization of $B$ by the multiplicatively closed subset $B\nz$ is
denoted by $Q(B)$, and called the {\em total ring of fractions} of $B$.

\paragraph\label{RFS.par}
Let $R$ be a commutative ring, $F$ an affine flat $R$-group scheme, and $S$ an
$F$-algebra.
Let $\omega:S\rightarrow S\otimes R[F]$ be the coaction,
where $R[F]$ denotes the coordinate ring of $F$.
Then $\omega$ is a flat ring homomorphism.
So $\omega': Q(S)\rightarrow Q(S\otimes R[F])$ is induced.
On the other hand, $\iota: S\rightarrow S\otimes R[F]$ 
given by $\iota(s)=s\otimes 1$ gives
$\iota':Q(S)\rightarrow Q(S\otimes R[F])$.
The kernel $\Ker(\iota'-\omega')\subset Q(S)$ is a subring of $Q(S)$.
We denote this subring by $Q(S)^F$ (this notation does {\em not} mean
that $F$ acts on $Q(S)$).
It is easy to see that $Q(S)^\times\cap Q(S)^F=(Q(S)^F)^\times$.
In particular, if $S$ is an integral domain, then $Q(S)^F$ is a subfield
of $Q(S)$, see \cite[Definition~6.1]{Mukai}.
Note also that $Q(S)^F\cap S=S^F$.

\begin{lemma}\label{field-abstract-grp.thm}
Let $S$ be a commutative $G$-algebra.
In general, $Q(S)^G\subset Q(S)^{G(k)}$.
If $G(k)$ is dense in $G$, 
then $Q(S)^G=Q(S)^{G(k)}$, where the right hand side is the 
ring of invariants of $Q(S)$ under the action of the abstract group $G(k)$.
\end{lemma}

\begin{proof} 
Let $a/b\in Q(S)$, where $a\in S$ and $b\in S\nz$.
Set $F=\omega(a)(b\otimes 1)-\omega(b)(a\otimes 1)$,
where $\omega:S\rightarrow S\otimes k[G]$ is the coaction.

Assume that $a/b\in Q(S)^G$, that is, $F=0$.
For $g\in G(k)$, let $\varphi_g:S\otimes k[G]\rightarrow S$ be the
$k$-algebra map given by $\varphi_g(f\otimes h)=h(g)f$.
Then for $f\in S$, we have that $g\cdot f=\varphi_g\omega(f)$.
As $F=0$, $\varphi_gF=(ga)b-(gb)a=0$ for any $g\in G(k)$.
Hence $ga/gb=a/b$ for any $g\in G(k)$, and $a/b\in Q(S)^{G(k)}$.

We prove the second assertion.
Assume that $a/b\in Q(S)^{G(k)}$.
In other words, $\varphi_gF=0$ for any $g\in G(k)$.
As $\Psi:k[G]\rightarrow (kG(k))^*$ given by $\Psi(f)(g)=f(g)$ is
injective by the density of $G(k)$ in $G$, the composite map
\[
\xi: S\otimes k[G]\specialarrow{1_S\otimes \Psi}
S\otimes (kG(k))^* \specialarrow{\zeta}
\Hom_k(kG(k),S)
\]
is injective, where $\zeta$ is the injective $k$-linear map
given by $\zeta(f\otimes \rho)(g)=(\rho(g))f$.
As $\xi(F)(g)=\varphi_g(F)=0$ for any $g\in G(k)$, $\xi(F)=0$.
So $F=0$ by the injectivity of $\xi$.
\end{proof}

\paragraph\label{Q_F(S).par}
Let the notation be as in (\ref{RFS.par}).
Let $\Omega=\Omega(S)$ be the set of $R$-submodules $M$ of $Q(S)$
such that $\omega'(M)\subset M\otimes_R R[F]$.
We define $Q_F(S):=\sum_{M\in\Omega}M\subset Q(S)$,
and we call $Q_F(S)$ the {\em $F$-total ring of fractions}.
Note that $Q_F(S)$ is the largest element of $\Omega$.
It is easy to see that $Q_F(S)$ is an $R$-subalgebra of $Q(S)$.
Note that $\omega'|_{Q_F(S)}:Q_F(S)\rightarrow Q_F(S)\otimes R[F]$ 
makes $Q_F(S)$ an $F$-algebra, and the inclusion $S\hookrightarrow
Q_F(S)$ is $F$-linear.
Note also that $Q_F(S)^F=Q(S)^F$.
If $S'\subset S$ is an $F$-subalgebra such that $Q(S')=Q(S)$,
then $Q_F(S')=Q_F(S)$ almost by definition.
In particular, $Q_F(Q_F(S))=Q_F(S)$.

\paragraph Let $R$ be a commutative ring, $F$ a flat $R$-group scheme, 
and $R'$ an $R$-algebra on which $F$ acts trivially.
Then we can identify an $R'\otimes_R F$-module with an $(F,R')$-module
using the canonical
isomorphism $M\otimes_{R'}(R'\otimes_R R[F])\cong M\otimes_R
R[F]$.
It is easy to see that $M^{R'\otimes_R F}=M^F$.
Similarly, to say that $S$ is an $R'$-algebra $F$-algebra such that 
the canonical map $R'\rightarrow S$ is an $F$-algebra map is the same
as to say that $S$ is an $R'\otimes_RF$-algebra.
If so, we have that $Q_F(S)=Q_{R'\otimes_R F}(S)$.

\begin{lemma}\label{M.thm}
Let $R$, $F$, and $S$ be as in {\rm(\ref{RFS.par})}.
Let $M$ be an $S$-submodule of $Q(S)$ such that $M\supset S$.
Then the following are equivalent.
\begin{description}
\item[(i)] $M\in\Omega$.
That is, $\omega'(M)\subset M\otimes_R R[F]$.
\item[(ii)] There is an $(F,S)$-module structure of $M$ such that
the inclusion $S\hookrightarrow M$ is $F$-linear.
\item[(iii)] There is a unique $(F,S)$-module structure of $M$ such that
the inclusion $S\hookrightarrow M$ is $F$-linear.
\item[(iv)] $M$ is an $(F,S)$-submodule of $Q_F(S)$.
\end{description}
In this case, the unique $(F,S)$-module structure of $M$ 
as in {\bf(iii)} is given by
$\omega'|_M:M\rightarrow M\otimes_R R[F]$.
It is also the induced submodule structure coming from {\bf(iv)}.
\end{lemma}

\begin{proof} 
{\bf (i)$\Rightarrow$(ii)}.
Consider the diagram
\begin{equation}\label{commutative.eq}
\xymatrix{
S \ar[d]^{\omega} \ar@{^{(}->}[r] \ar@{}[dr]|{\text{(a)}} & 
M \ar@{^{(}->}[r] \ar[d]^{\omega_1} \ar@{}[dr]|{\text{(b)}} & 
Q(S) \ar[d]^{\omega'} \\
S\otimes_R R[F] \ar@{^{(}->}[r] &
M\otimes_R R[F] \ar@{^{(}->}[r] &
Q(S\otimes_R R[F]),
}
\end{equation}
where $\omega_1=\omega'|_M$.
By the definition of $\omega'$, (a)$+$(b) is commutative.
On the other hand, by the definition of $\omega_1$, (b) is commutative.
As $M\otimes_R R[F]\hookrightarrow Q(S\otimes_R R[F])$ is injective, 
(a) is commutative.
In other words, $S\hookrightarrow M$ is $F$-linear.
The commutativity of (a)$+$(b) shows that $\omega'$ is an $S$-algebra map,
where the $S$-algebra structure of $Q(S\otimes_R R[F])$ is given by the
composite map
\[
S\xrightarrow{\omega} S\otimes_R R[F]\hookrightarrow Q(S\otimes_R R[F]).
\]
By the commutativity of (b), it is easy to see that $\omega_1$ is $S$-linear.
This shows that $M$ is an $(R[F],S)$-Hopf module.
In other words, $M$ is an $(F,S)$-module.

{\bf (ii)$\Rightarrow$(iii), (i)}
Let us consider the diagram (\ref{commutative.eq}), where $\omega_1$ is
the given comodule structure of $M$.
Then (a) is commutative.
Let $b/a\in M$, where $b\in S$ and $a\in S\nz$.
Then $\omega(a)\omega_1(b/a)=\omega_1(a\cdot(b/a))=\omega(b)$.
As $\omega(a)$ is a nonzerodivisor on $M\otimes R[F]$,
$\omega_1(b/a)=\omega(b)/\omega(a)=\omega'(b/a)$, and $\omega_1=\omega'|_M$
is unique.
The commutativity of (b) shows that $\omega'(M)\subset M\otimes_R R[F]$.

{\bf (iii)$\Rightarrow$(ii)} is trivial.

{\bf (i)$\Leftrightarrow$(iv)} The coaction $\omega:Q_F(S)\rightarrow Q_F(S)
\otimes_R R[F]$ is the restriction of $\omega'$.
So the condition {\bf(i)} says that $M$ is a subcomodule of $Q_F(S)$.
As $M$ is an $S$-submodule of $Q_F(S)$, it is an $(F,S)$-submodule if
and only if {\bf (i)} holds.

Now the equivalence of {\bf (i)--(iv)} has been proved.
The unique $(F,S)$-module structure of $M$ as in {\bf (iii)} is given by
$\omega'|_M$ by the proof of {\bf (ii)$\Rightarrow$(iii)}.
This agrees with the induced submodule structure coming from {\bf (iv)},
since the coaction of $Q_F(S)$ is also the restriction $\omega'|_{Q_F(S)}$
of $\omega'$.
\end{proof}

\begin{lemma}\label{omega-inverse.thm}
Let $R$, $F$, and $S$ be as in {\rm(\ref{RFS.par})}.
Then $(\omega')^{-1}(Q(S)\otimes_R R[F])=Q_F(S)$.
\end{lemma}

\begin{proof}
Set $C$ to be the left hand side.
As $\omega'(Q_F(S))\subset Q_F(S)\otimes_R R[F]$,
$Q_F(S)\subset C$.
In particular, $S\subset C$.

Consider the composite map
\[
\rho:C\xrightarrow{\omega'} Q(S)\otimes_R R[F]\xrightarrow{1\otimes \Delta}
Q(S)\otimes_R R[F]\otimes_R R[F]
\subset Q(S\otimes_R R[F]\otimes_R R[F]),
\]
where $\Delta : R[F]\rightarrow R[F]\otimes_R R[F]$ is the coproduct.
Also consider the map
\[
\rho':C\xrightarrow{\omega'}Q(S)\otimes_R R[F]
\xrightarrow{\omega'\otimes1}
Q(S\otimes_R R[F])\otimes_R R[F]
\subset Q(S\otimes_R R[F]\otimes_R R[F]).
\]
Then $\rho$ and $\rho'$ are $R$-algebra maps, and 
$\rho|_S=\rho'|_S$ by the coassociativity law on the $R[F]$-comodule $S$.
It follows easily that $\rho=\rho'$.
This shows that 
\[
\omega'(C)\subset (\omega'\otimes 1)^{-1}(Q(S)\otimes_R R[F]
\otimes_R R[F])
=C\otimes_R R[F].
\]
It follows immediately that $C\subset Q_F(S)$.
Hence $C=Q_F(S)$, as desired.
\end{proof}

\begin{corollary}
Let $R$, $F$, and $S$ be as in {\rm(\ref{RFS.par})}.
Assume that $S$ is Noetherian and $F$ is finite over $R$.
Then $Q_F(S)=Q(S)$.
\end{corollary}

\begin{proof}
Note that every maximal ideal of $Q(S)$ is an associated prime of zero.
As $R[F]$ is finite flat over $R$, every maximal ideal of $Q(S)\otimes_R
R[F]$ is an associated prime of zero.
So $Q(S)\otimes_R R[F]=Q(S\otimes_R R[F])$.
Hence 
\[
Q_F(S)=(\omega')^{-1}(Q(S)\otimes_R R[F])
=(\omega')^{-1}(Q(S\otimes_R R[F]))=Q(S).
\]
\end{proof}

\begin{lemma}\label{pre-normalization.thm}
Let $X_0$ be a scheme, $F$ a smooth $X_0$-group scheme of finite type,
and $X$ a Noetherian reduced $X_0$-scheme with an action of $F$.
Let $\varphi:X'\rightarrow X$ be the normalization of $X$.
Then there is a unique $F$-action of $F$ on $X'$ such that $\varphi$ is
an $F$-morphism.
\end{lemma}

\begin{proof}
Note that $F\times_{X_0}X'$ is Noetherian normal, and 
the composite 
\[
F\times_{X_0}X'\xrightarrow{1\times \varphi}
F\times_{X_0}X \xrightarrow{a}X
\]
maps each connected component of $F\times_{X_0}X'$ 
dominatingly to an irreducible component of $X$, where $a$ is the action
of $F$ on $X$.
Thus by the universality of the normalization, 
there is a unique map $a':F\times_{X_0}X'\rightarrow X'$ such
that $\varphi\circ a'=a\circ(1\times\varphi)$.

It remains to prove that $a'$ is an action of $F$ on $X'$.
The diagram
\[
\xymatrix{
X' \ar[r]^-{e\times 1} \ar[d]^{\varphi} & 
F\times_{X_0}X' \ar[r]^-{a'} \ar[d]^{1\times\varphi} &
X' \ar[d]^\varphi \\
X \ar[r]^-{e\times 1} &
F\times_{X_0}X \ar[r]^-a &
X
}
\]
is commutative, and $a \circ (e\times 1)\circ\varphi=\varphi$.
Thus $\varphi\circ a'\circ (e\times 1)=\varphi\circ \id_{X'}$.
By the uniqueness, $a'\circ (e\times 1)=\id_{X'}$, the unit law holds.

The morphisms $F\times_{X_0}F\times_{X_0}X'\rightarrow X$ given by
$(f,g,x')\mapsto f(g (\varphi(x')))$ and
$(f,g,x')\mapsto (fg)(\varphi(x'))$ agree.
So $\varphi((fg)x')=\varphi(f(gx'))$.
By the uniqueness, $(fg)x'=f(gx')$.
The associativity also holds, and $a'$ is an action, as desired.
\end{proof}

\begin{corollary}\label{normalization.thm}
Let $R$ be a commutative ring, $F$ an affine smooth $R$-group
scheme of finite type, and $S$ a Noetherian reduced $F$-algebra.
Then the integral closure $S'$ of $S$ in $Q(S)$
has a unique $F$-algebra structure such that the inclusion 
$S\hookrightarrow S'$ 
is an $F$-algebra map.
In particular, $S\subset S'\subset Q_F(S)=Q_F(S')$.
\end{corollary}

\begin{proof}
The first assertion is obvious by Lemma~\ref{pre-normalization.thm}.
The second assertion follows from the first and Lemma~\ref{M.thm}.
\end{proof}

\begin{lemma}\label{colon.thm}
Let $R$ be a commutative ring, $F$ an affine flat $R$-group scheme,
and $S$ a Noetherian $F$-algebra.
Then $Q_F(S)=\bigcup_I S:_{Q(S)}I$, where $I$ runs through the all
$F$-ideals of $S$ containing a nonzerodivisor.
For each $I$, $S:_{Q(S)}I$ is an $(F,S)$-submodule of $Q_F(S)$, and
the canonical map $\varphi:S:_{Q(S)}I\rightarrow\Hom_S(I,S)$ given by
$\varphi(\alpha)(a)=\alpha a$ is an $(F,S)$-linear isomorphism.
\end{lemma}

\begin{proof} 
Let $I$ be an $F$-ideal of $S$ containing a nonzerodivisor.
The diagram
\[
\xymatrix{
S \ar@{^{(}->}[r] \ar[d]^\varphi_\cong & S:_{Q(S)}I \ar[d]^\varphi_\cong \\
\Hom_S(S,S) \ar[r]^{\iota^*} & \Hom_S(I,S)
}
\]
is commutative, where $\varphi(\alpha)(a)=\alpha a$, and
$\iota^*$ is the pull back with respect to the inclusion map $\iota:
I\hookrightarrow S$.
As in \cite[(I.5.3.4)]{Hashimoto}, $\Hom_S(S,S)$ and $\Hom_S(I,S)$ are
$(F,S)$-modules, and $\iota^*$ is an $(F,S)$-linear map.
Note that the isomorphism
$\varphi: S\rightarrow \Hom_S(S,S)$ is $(F,S)$-linear.
By the commutativity of the diagram, $S:_{Q(S)}I$ possesses an
$(F,S)$-module structure such that the inclusion $S\hookrightarrow
S:_{Q(S)}I$ is $(F,S)$-linear, and $\varphi:S:_{Q(S)}I\rightarrow
\Hom_S(I,S)$ is an $(F,S)$-isomorphism.
By Lemma~\ref{M.thm}, $S:_{Q(S)}I\subset Q_F(S)$ is an $(F,S)$-submodule.

Next we show that $Q_F(S)\subset \bigcup_I S:_{Q(S)}I$.
Let $\alpha\in Q_F(S)$.
Then there is an $S$-finite $(F,S)$-submodule $M$ of 
$Q_F(S)$ containing $\alpha$
and $1$ by \cite[(I.5.3.11)]{Hashimoto}.
Then $I:=S:_S M$ is an ideal of $S$ containing a nonzerodivisor.
Being the kernel of the $F$-linear map $S\rightarrow \Hom_S(M,M/S)$, 
$I$ is an $F$-ideal, and $\alpha\in S:_{Q(S)}I$.
\end{proof}

\begin{corollary}\label{colon-cor.thm}
Let $R$, $F$, and $S$ be as in {\rm Lemma~\ref{colon.thm}}.
Let $I$ and $J$ be $F$-ideals of $S$.
If $J$ contains a nonzerodivisor, then $I:_{Q(S)}J$ is an
$(F,S)$-submodule of $Q_F(S)$.
\end{corollary}

\begin{proof} 
As the diagram
\[
\xymatrix{
I:_{Q(S)}J \ar@{^{(}->}[r] \ar[d]^\varphi_{\cong} &
S:_{Q(S)}J \ar[d]^\varphi_{\cong} \\
\Hom_S(J,I) \ar@{^{(}->}[r] & \Hom_S(J,S)
}
\]
is commutative, $\varphi:S:_{Q(S)}J\rightarrow \Hom_S(J,S)$ is an
isomorphism of $(F,S)$-modules, 
$\varphi: I:_{Q(S)}J\rightarrow \Hom_S(J,I)$ is bijective, 
and $\Hom_S(J,I)$ is an
$(F,S)$-submodule of $\Hom_S(J,S)$, $I:_{Q(S)}J$ is an
$(F,S)$-submodule of $S:_{Q(S)}J$, which is an $(F,S)$-submodule of
$Q_F(S)$.
Hence $I:_{Q(S)}J$ is an $(F,S)$-submodule of $Q_F(S)$.
\end{proof}

\paragraph Let $S$ be a Noetherian ring.
For an open subset $U\subset \Spec S$, the canonical map $S\rightarrow
\Gamma(U,\O_{\Spec S})$ is injective if and only if any (or equivalently, 
some) ideal $I\subset S$ such that $D(I):=\Spec S\setminus V(I)=U$ contains a
nonzerodivisor.
Thus $\Gamma(U,\O_{\Spec S})$ is identified with a subring of $Q(S)$, and
$Q(S)$  is identified with $\indlim \Gamma(U,\O_{\Spec S})$, where 
$U$ runs through all the open subsets of $\Spec S$ such that
$S\rightarrow \Gamma(U,\O_{\Spec S})$ is injective.
For $f\in Q(S)$, there is a largest $U$ such that $f\in\Gamma(U,\O_{
\Spec S})$.
We denote this $U$ by $U(f)$, and call it the domain of definition of $f$.

\paragraph If $S$ is not Noetherian, even if $S\rightarrow \Gamma(U,\O_{\Spec
S})$ is injective, the ring of sections 
$\Gamma(U,\O_{\Spec S})$ may not be a subring of $Q(S)$.
Let $k$ be a countable field, and 
$S=k[x_1,y_1,x_2,y_2,\ldots]/J$, where $J$ is the ideal generated by
$x_iy_j$ with $j\geq i$, and $y_iy_j$ with $j\geq i$.
Let $I:=(x_1,y_1,x_2,y_2,\ldots)$, and $U:=D(I)$.
Then it is easy to see that $\Gamma(U,\O_{\Spec S})
=\projlim S/J_i$ is uncountable, 
where 
$J_i$ is generated by $J$ and $\{y_j\mid j\geq i\}$.
On the other hand, 
$Q(S)$ is countable, and so $\Gamma(U,\O_{\Spec S})$ cannot be
a subring of $Q(S)$.

\begin{lemma}\label{UD.thm}
Let $S$ be a Noetherian ring, and $h\in Q(S)$.
For $f\in S$, we have $h\in \bigcup_{n\geq 0}S:_{Q(S)}f^n$
if and only if $U(h)\supset D(f)=\Spec S[1/f]$.
\end{lemma}

\begin{proof}
Assume that $U(h)\supset D(f)$.
Then $h/1\in S[1/f]$ makes sense, and we can write $h/1=a/f^n$ in $S[1/f]$.
Then $f^nh-a=0$ in $Q(S)[1/f]$.
So $f^m(f^nh-a)=0$ in $Q(S)$ for some $m$.
So $h\in S:_{Q(S)}f^{n+m}$.

Conversely, assume that $h\in Q(S)$ and $f^n h=a\in S$.
Then $h=a/f^n$ in $\Gamma(U(h)\cap D(f),\O_{\Spec S})$.
As $\O_{\Spec S}$ is a sheaf, there exists some $\beta
\in\Gamma(U(h)\cup D(f),\O_{\Spec S})$ such that $\beta$ in 
$\Gamma(U(h),\O_{\Spec S})$ is $h$, and $\beta $ in 
$S[1/f]$ is $a/f^n$.
By the maximality of $U(h)$, $U(h)=U(h)\cup D(f)$.
That is, $U(h)\supset D(f)$.
\end{proof}

\begin{lemma}[cf.~{\cite[Appendix, Proposition~4]{Hartshorne}}]\label{aho.thm}
Let $S$ be a Noetherian ring, and $I$ an ideal of $S$ which contains a
nonzerodivisor.
Set $U:=D(I)=\Spec S\setminus V(I)$.
Then as a subset of $Q(S)$, we have
\[
\bigcup_{n\geq 0}S:_{Q(S)}I^n=\Gamma(U,\Cal O_{\Spec S}).
\]
\end{lemma}

\begin{proof}
Let $I=(f_1,\ldots,f_r)$.
By Lemma~\ref{UD.thm}, 
\begin{multline*}
\Gamma(U,\Cal O_{\Spec S})
=
\{h\in Q(S)\mid U(h)\supset U\}
=\bigcap_{i=1}^r \{h\in Q(S)\mid U(h)\supset D(f_i)\}\\
=
\bigcap_{i=1}^r \bigcup_{n\geq 0}S:_{Q(S)}f_i^n
=\bigcup_{n\geq 0}S:_{Q(S)}I^n.
\end{multline*}
\end{proof}

\begin{lemma}
Let $R$, $F$, and $S$ be as in {\rm Lemma~\ref{colon.thm}}.
Let $U$ be an open subset of $\Spec S$.
Then $U$ is $F$-stable if and only if $U=D(I)$ for some $F$-stable ideal
$I$ of $S$.
\end{lemma}

\begin{proof}
The \lq if' part is obvious.
We prove the \lq only if' part.
Let $U$ be an $F$-stable open set.
Then $U=D(J)$ for some ideal $J$ of $S$.
Then $U=D(J^*)$, see \cite[(8.3)]{HO}, where $J^*$ is the largest 
$F$-stable ideal of $S$ contained in $J$.
\end{proof}

\begin{lemma}
Let $R$, $F$, and $S$ be as in {\rm Lemma~\ref{colon.thm}}.
Then $Q_F(S)=\indlim \Gamma(U,\O_{\Spec S})$, where $U$ runs through all
the $F$-stable open subsets such that $S\rightarrow\Gamma(U,\O_{\Spec S})$ 
is injective.
In particular, if there is a unique smallest $F$-stable 
open subset $U$ such that
$S\rightarrow\Gamma(U,\O_{\Spec S})$ is injective, 
then $Q_F(S)=\Gamma(U,\O_{\Spec S})$.
\end{lemma}

\begin{proof}
Note that for each $U$, $\Gamma(U,\O_{\Spec S})$ is an $F$-algebra, and
the map 
$S\rightarrow \Gamma(U,\O_{\Spec S})$ is an $F$-algebra map.
Moreover, for $U\supset V$, 
$$\Gamma(U,\O_{\Spec S})\rightarrow
\Gamma(V,\O_{\Spec S})$$ is an $F$-algebra map.
It follows that $\indlim\Gamma(U,\O_{\Spec S})$ is an $F$-algebra, and 
the canonical map $S\rightarrow\indlim\Gamma(U,\O_{\Spec S})$ is an
$F$-algebra map.
So we have that $\indlim\Gamma(U,\O_{\Spec S})\subset Q_F(S)$ by
Lemma~\ref{M.thm}.

We prove $Q_F(S)\subset\indlim\Gamma(U,\O_{\Spec S})$.
By Lemma~\ref{colon.thm}, it suffices to show that $S:_{Q(S)}I
\subset \Gamma(D(I),\O_{\Spec S})$ for any $F$-ideal $I$ containing 
a nonzerodivisor.
This is Lemma~\ref{aho.thm}.

The last assertion is trivial.
\end{proof}

\begin{lemma}\label{Krull.thm}
Let $R$ be a commutative ring, $F$ a flat affine $R$-group scheme, 
and $S$ a Noetherian normal $F$-algebra.
Then 
$f\in Q(S)$ lies in $Q_F(S)$ if and only
if $f/1\in S_P$ for any height one prime ideal of $S$ such that
$P^*$ does not contain a nonzerodivisor, where $P^*$ is the largest
$F$-ideal of $S$ contained in $P$.
In particular, 
$Q_F(S)$ is a finite direct product of Krull domains.
In particular, $Q_F(S)$ is integrally closed in $Q(S)$.
\end{lemma}

\begin{proof}
We prove the \lq if' part.
Let $P_1,\ldots,P_n$ be the 
height one prime ideals such that $f/1\notin S_{P_i}$.
For each $i$, there exists some $m(i)$ 
such that $P_i^{m(i)}(f/1)\in S_{P_i}$.
Letting $I:=\prod_i (P_i^*)^{m(i)}$, $I(f/1)\in S_P$ for 
any height one prime ideal of $S$.
Thus $If\subset S$, and $I$ is an $F$-ideal of $S$ containing a nonzerodivisor.
Hence, $f\in S:_{Q(S)}I\subset Q_F(S)$.

We prove the \lq only if' part.
So let $f\in Q_F(S)$.
Then by Lemma~\ref{colon.thm}, $f\in S:_{Q(S)}I$ for some 
$F$-ideal $I$ of $S$ containing a nonzerodivisor.
Let $P$ be a height one prime of $S$ such that $(f/1)\notin S_P$.
This implies $I\subset P$, since $If\in S$.
As $I$ is an $F$-ideal, $I\subset P^*$, and $P^*$ contains a nonzerodivisor.
\end{proof}

\begin{corollary}\label{nagata.thm}
Let $R$ be a commutative ring, $F$ an affine smooth $R$-group scheme of
finite type, and $S$ a Noetherian reduced Nagata $F$-algebra.
Then $Q_F(S)$ is a finite direct product of Krull domains.
\end{corollary}

\begin{proof} 
Note that $S'$ is a Noetherian normal $F$-algebra.
On the other hand, $Q_F(S)=Q_F(S')$.
By Lemma~\ref{Krull.thm}, the assertion follows.
\end{proof}

\begin{lemma}\label{associated.thm}
Let $k$ be a field, and $H$ an affine algebraic $k$-group scheme of finite
type.
Let $X$ be a Noetherian $H$-scheme, and $Y$ a primary 
\(i.e., irreducible and $(S_1)$\) closed subscheme of $X$.
Let $Y^*$ (resp.\ $Y'$) denotes the smallest closed
$H$-subscheme \(resp.\ $H^\circ$-subscheme\) of $X$ containing $Y$.
Then $Y^*$ does not have an embedded component, and
$Y'\red$ is an irreducible component of $Y^*$.
\end{lemma}

\begin{proof}
As $Y$ is primary, $Y^*$ is $H$-primary, and it does not have an
embedded component by \cite[(6.2)]{HM}.
As $Y$ is primary, $Y'$ is primary by \cite[(6.23)]{HM}.
In particular, $Y'\red$ is integral.
Obviously, $Y'\red\subset Y^*$.

For a $k$-scheme $W$, let $\bar W$ denote $\bar k\otimes_k W$.
Then $Y^*$ is the scheme theoretic image of the composite
\[
\bar H\times_{\bar k} \bar Y
\xrightarrow{\alpha} \bar X\xrightarrow{p} X,
\]
where $\alpha$ is the action, and $p$ is the projection.

Note that there exist some $\bar k$-valued points 
$h_1,\ldots,h_r$ of $\bar H$ such that $\bar H=\coprod_i h_i \bar H^\circ$,
where $h_1=e$ is the identity element.
Let $W_i$ be the scheme theoretic image of $\alpha:
h_i\bar H^\circ\times_{\bar k}\bar Y\rightarrow \bar X$, and $V_i$ 
be the scheme theoretic image of $p:W_i\rightarrow X$.
Note that $(V_1)\red=Y'\red$.
The local rings $\O_{X,(V_i)\red}$ and $\O_{\bar X,(W_i)\red}$ have the
same dimension.
Indeed, $\O_{X,(V_i)\red}\rightarrow\O_{\bar X,(W_i)\red}$ is a flat 
local homomorphism, as $p$ is flat.
So
\[
\dim \O_{\bar X,(W_i)\red}=\dim \O_{X,(V_i)\red}
+\dim \O_{\bar X,(W_i)\red}\otimes_{\O_{X,(V_i)\red}}\kappa(v_i)
\]
where $v_i$ is the generic point of $V_i$.
But $p^{-1}(v_i)$ is zero-dimensional, as $p$ is an integral morphism.

As the action of $h_i$ induces an isomorphism
$\O_{\bar X,(W_i)\red}\cong \O_{\bar X,(W_1)\red}$, 
we have
\[
\dim \O_{X,(V_i)\red}=
\dim \O_{\bar X,(W_i)\red}=
\dim \O_{\bar X,(W_1)\red}=
\dim \O_{X,(V_1)\red}.
\]
This shows that $(V_i)\red$ does not strictly contain $(V_1)\red=Y'\red$.
As $Y^*=\bigcup_i V_i$, we are done.
\end{proof}

\begin{example}
Let $R=k$ be a field, $F=H$ be of finite type over $k$.
Let $S$ be an $H$-algebra which is a Noetherian
normal domain.
Then for a height one prime ideal $P$ of $S$, $P^*$, the largest 
$H$-ideal of $S$ contained in $P$, contains a nonzerodivisor
(or equivalently, nonzero) if and only if $P'$, the largest 
$H^\circ$-ideal of $S$ contained in $P$, contains a nonzerodivisor.

As $P'$ is a primary ideal by \cite[(6.23)]{HM} and $P'\subset P$, 
it is $P$-primary or $0$-primary.
So $P'\neq 0$ if and only if $P'$ is $P$-primary.

If $P'$ is $P$-primary, as $P^*$ does not have an embedded prime
and $P$ is an associated prime of $P^*$ by
Lemma~\ref{associated.thm}, $0$ is not an associated prime of $P^*$.
Thus $P^*\neq 0$.
If $P'$ is not $P$-primary, $P'=0$, and hence $P^*=0$.

By Lemma~\ref{Krull.thm}, we have $Q_H(S)=Q_{H^\circ}(S)$.
This shows that the finite group scheme $H/H^\circ$ acts on $Q(S)^{H^\circ}$ 
in a natural way.
Indeed, 
\[
Q(S)^{H^\circ}=Q_{H^\circ}(S)^{H^\circ}=Q_H(S)^{H^\circ}.
\]
As $H$ acts on $Q_H(S)$, $H/H^\circ$ acts on $Q_H(S)^{H^\circ}$.
Moreover, we have
\[
(Q(S)^{H^\circ})^{H/H^\circ}=(Q_H(S)^{H^\circ})^{H/H^\circ}=Q_H(S)^H
=Q(S)^H.
\]
\end{example}

\begin{lemma}
Let $\varphi: H'\rightarrow H$ be a surjective homomorphism of
affine algebraic $k$-group schemes.
Let $S$ be a Noetherian normal $H$-algebra domain.
Then $Q_{H'}(S)=Q_H(S)$.
\end{lemma}

\begin{proof} 
Let $P$ be a height one prime ideal of $S$.
Let $P^*$ and $P'$ denote the largest $H$-stable and $H'$-stable
ideal contained in $P$, respectively.
By assumption, $\sqrt{P^*}=\sqrt{P'}$.
So $P^*$ is nonzero if and only if $P'$ is nonzero.
By Lemma~\ref{Krull.thm}, we are done.
\end{proof}

\begin{lemma}\label{ST.thm}
Let $R$ be a commutative ring, and $F$ an affine flat $R$-group scheme.
Let $\varphi:S\rightarrow T$ be an $F$-algebra map.
Assume that $\varphi(S\nz)\subset T\nz$, and let
$Q(\varphi):Q(S)\rightarrow Q(T)$ be the induced map.
Then $Q(\varphi)(Q_F(S))\subset Q_F(T)$, and
$Q_F(\varphi):=Q(\varphi)|_{Q_F(S)}:Q_F(S)\rightarrow Q_F(T)$ is an
$F$-algebra map.
\end{lemma}

\begin{proof} 
Let $\omega_S$ and $\omega_T$ respectively denote the
coaction $S\rightarrow S\otimes_R R[F]$ and $T\rightarrow T\otimes_R R[F]$.
Let $Q'_S$ be the localization of $S\otimes_R R[F]$ by the multiplicatively
closed subset 
\[
C(S):=\omega(S\nz)\cdot (S\nz\otimes 1)
=\{\omega(a)(b\otimes 1)\mid a,b\in S\nz\}.
\]
We define $C(T)$ and $Q'_T$ similarly.
It is easy to see that $(\varphi\otimes 1)(C(S))\subset C(T)$, and
$Q'(\varphi):Q'_S\rightarrow Q'_T$ is induced.
Let $\omega_S':Q(S)\rightarrow Q'_S$ and
$\omega_T':Q(T)\rightarrow Q'_T$ be the induced
maps, induced by $\omega_S$ and $\omega_T$, respectively.
Then
\begin{multline*}
\omega_T'(Q(\varphi)(Q_F(S)))=Q'(\varphi)\omega_S'(Q_F(S))\\
\subset Q'(\varphi)(Q_F(S)\otimes R[F])\subset
(Q(\varphi)(Q_F(S)))\otimes R[F].
\end{multline*}
So $Q(\varphi)(Q_F(S))\subset Q_F(T)$.

Next, the five faces except for the top one of the cube
\[
\xymatrix{
Q_F(S) \ar[rr]^{Q_F(\varphi)} \ar[dr]^{\omega_{1,S}} \ar@{^{(}->}[dd] & &
Q_F(T) \ar[dr]^{\omega_{1,T}} \ar@{^{(}->}'[d][dd] & \\
 & Q_F(S)\otimes R[F] \ar[rr]^{Q_F(\varphi)\otimes 1~~~~~~~~~~~}
\ar@{^{(}->}[dd] & &
Q_F(T)\otimes R[F] \ar@{^{(}->}[dd]^h \\
Q(S) \ar '[r][rr]^{Q(\varphi)~~~~~~~} \ar[dr]^{\omega_S'} & & 
Q(T) \ar[dr]^{\omega_T'} & \\
 & Q'_S \ar[rr]^{Q'(\varphi)} & &
Q'_T
}
\]
are commutative, where $\omega_{1,S}$ and $\omega_{1,T}$ are
the coaction of $Q_F(S)$ and $Q_F(T)$, respectively.
As $h$ is injective, the top face is also commutative,
and hence $Q_F(\varphi):Q_F(S)\rightarrow Q_F(T)$ is $F$-linear.
Being a restriction of the ring homomorphism $Q(\varphi)$, it is
a ring homomorphism, as desired.
\end{proof}

\paragraph A field extension $L/K$ is said to be primary (resp.\ regular), if
$K\sep\otimes_K L$ (resp.\ $\bar K\otimes_K L$) 
is a field, where $K\sep$ (resp.\ $\bar K$) 
is the separable closure (resp.\ algebraic closure) of $K$.
See \cite[Expos\'e~14]{Cartan}.

\begin{lemma}\label{primary.thm}
Let $S$ be an $H$-algebra domain.
If $H$ is connected, then 
the field extension $Q(S)/Q(S)^H$ is primary.
\end{lemma}

\begin{proof}
Replacing $S$ by $Q_H(S)$, we may assume that $S=Q_H(S)\supset Q(S)^H$.
Replacing $k$ by $Q(S)^H$, we may assume that $k=Q(S)^H$.
It suffices  to show that for any $H$-subalgebra $T$ of $S$ which
is a finitely generated domain, the assertion of the lemma
is true for $T$, since $Q(T)^H=k=Q(S)^H$, and $Q(S)=\indlim Q(T)$.
Replacing $S$ by $T$, we may assume that $k=Q(S)^H$, and
$S$ is a finitely generated domain (possibly losing the additional
assumption $S=Q_H(S)$).

Note that $S\sep:=k\sep\otimes_k S$ 
is a reduced algebra of finite type over $k\sep$.
As $k\sep\otimes_k Q(S)$ is essentially of finite type over $k\sep$, it
is Noetherian.
As it is also integral over the field $Q(S)$, it is zero-dimensional.
Being a reduced Artinian ring, $k\sep\otimes_k Q(S)$ is a finite direct 
product of fields.
So it agrees with $Q(S\sep)$.
Thus it is easy to see that $Q(S\sep)^H=k\sep\otimes_k Q(S)^H
=k\sep$.
Let $P_1,\ldots,P_n$ be the minimal primes of $S\sep$.

Since $P_i^*$ is $P_i$-primary by \cite[(6.23)]{HM}, 
we have $P_i^*=P_i$, as $S\sep$ is reduced.
So each $P_i$ is $H$-stable.
Then $(\prod_{i=1}^n S\sep/P_i)^H=k\sep$.
As any idempotent of $(\prod_{i=1}^n S\sep/P_i)$ is $H$-invariant
by Lemma~\ref{idempotent.thm}, we have $n=1$.
That is, $S\sep$ is an integral domain.
So $Q(S\sep)=k\sep\otimes_k Q(S)$ is a field.
This shows that $Q(S)/k$ is a primary extension, as desired.
\end{proof}

\begin{lemma}\label{separable.thm}
Let $K$ be a field,
and $\Gamma$ an abstract group acting on $K$ as
automorphisms.
Then the extension $K/K^\Gamma$ is separable.
\end{lemma}

\begin{proof}
We may assume that $\charac(K)=p$ is positive.

Let $f_1,\ldots,f_n$ be elements of $K$ which are linearly independent 
over $K^\Gamma$.
Assume that $\sum_j\alpha_j f_j^p=0$ for $\alpha_j\in K^\Gamma$.
By Artin's lemma
\cite[\S7, n${}^\circ$1]{Bourbaki}, 
there exist some $g_1,\ldots,g_n\in\Gamma$
such that $\det(g_if_j)\neq 0$.
Then $\det(g_if_j^p)=\det(g_if_j)^p\neq 0$.
As $\sum_j \alpha_j g_if_j^p=0$ for all $i$, $\alpha_1=\cdots=\alpha_n=0$.
So $f_1^p,\ldots,f_n^p$ is linearly independent over $K^\Gamma$.
So $f_1,\ldots,f_n$ is linearly independent over $(K^\Gamma)^{-p}$
in $K^{-p}$.
So the canonical map $(K^\Gamma)^{-p}\otimes_{K^\Gamma}K\rightarrow 
K^{-p}$ is injective, and
being a subring of a field, $(K^\Gamma)^{-p}\otimes_{K^\Gamma}K$ is
reduced.
So $K/K^\Gamma$ is separable.
\end{proof}

\begin{lemma}\label{finite-separable.thm}
Let $H$ be a finite $k$-group scheme, and $S$ an $H$-algebra.
Then $S$ is integral over $S^H$.
\end{lemma}

\begin{proof}
As $S$ is a subring of $\bar S:=\bar k\otimes_k S$, it suffices to show that
$\bar S$ is integral over $S^H$.
As $(\bar k\otimes_k S)^H=\bar k \otimes_k S^H$ is integral over $S^H$, 
we may assume that $k$ is algebraically closed.
As $S^H=(S^{H^\circ})^{H/H^\circ}$, we may assume that either 
$H$ is infinitesimal or reduced.

The reduced case is well-known.
If $a\in S$, then $f(a)=0$ for $f:=\prod_{h\in H}(t-ha)\in S^H[t]$.

Assume that $H$ is infinitesimal.
Let $I$ be the nilradical of $k[H]$.
Note that $I=\Ker(\varepsilon)$, where $\varepsilon$ is the counit map
of $k[H]$.
If $I^{p^e}=0$ ($e\geq 0$), 
then $f^{p^e}=\varepsilon(f)^{p^e}$ in $k[H]$ for $f\in k[H]$.
So it is easy to see that $S^{p^e}\subset S^H\subset S$.
As $S$ is integral over $S^{p^e}$, it is also integral over $S^H$.
\end{proof}

\begin{lemma}\label{separable-regular.thm}
Let $S$ be a $G$-algebra which is a domain.
Then the field extension $Q(S)/Q(S)^G$ is separable.
If $G$ is connected, then the extension 
$Q(S)/Q(S)^G$ is regular.
\end{lemma}

\begin{proof}
By Lemma~\ref{primary.thm}, $Q(S)/Q(S)^G$ is primary.
As a primary separable extension is regular, it suffices to prove
the first assertion.

As in the proof of Lemma~\ref{primary.thm}, we may assume that
$S$ is finitely generated, and $Q(S)^G=k$.
Replacing $S$ by its normalization, we may assume that $S$ is normal.
It suffces to show that $S$ is geometrically reduced over $k$ (see
(\ref{separable.par})).
Set $K=S^{G^\circ}$.
Then by Lemma~\ref{finite-separable.thm}, $K$ is integral over 
$K^{G/G^\circ}=S^G=k$.
Hence $K$ is an algebraic extension field of $k$.
As $K\subset Q(S)$, $K$ is finite over $k$.
It suffices to prove that $K$ is separable over $k$, and $S$ is geometrically
reduced over $K$.
Thus we may assume that $G$ is either connected or finite.

Set $S\sep:=k\sep\otimes_k S$.
It suffices to show that $S\sep$ is geometrically reduced over $k\sep$.
Then $S\sep$ is geometrically reduced over $k$, and hence
$S$ is geometrically reduced over $k$, as desired.
Note that $S\sep=S_1\times\cdots\times S_n$ is a finite direct product of
normal domains which are finitely generated over $k\sep$.

First assume that $G$ is connected.
As $S\sep^G=k\sep$ does not have a nontrivial idempotent, 
$n=1$ by Lemma~\ref{idempotent.thm}.
So $Q(S\sep)=k\sep\otimes_k Q(S)$ is a field.
As
\[
Q(S\sep)^{G(k\sep)}=(k\sep\otimes_k Q(S))^G=k\sep\otimes_k Q(S)^G
=k\sep\otimes_k k=k\sep,
\]
the extension $Q(S\sep)/k\sep$ is separable by
Lemma~\ref{separable.thm}.

Next consider the case that $G$ is finite.
Then $Q(S\sep)^{\Gamma}=k\sep$ as above, where $\Gamma=G(k\sep)$.
Let $x\in Q(S\sep)$.
Set $H=\{g\in \Gamma\mid gx=x\}$.
Let $g_1,\ldots,g_r$ be the complete set of representatives of $G/H$.
Then $\Gamma$ permutes $g_1x,\ldots,g_rx$.
So $f(t)=\prod_i (t-g_ix)$ lies in $k\sep[t]$, and $f(x)=0$.
So $x$ is separable over $k\sep$.
So $Q(S\sep)$ is geometrically reduced over $k\sep$, as desired.
\end{proof}

\section{Factoriality of rings generated by semiinvariants}

\begin{lemma}\label{dvr.thm}
Let $R$ be a discrete valuation ring, $\varphi:R\rightarrow R$ a
ring automorphism, and $a\in Q(R)\setminus\{0\}$.
Then $\varphi(a)\cdot a^{-1}\in R^\times$ \(note that $\varphi$ can be
extended to an automorphism of $Q(R)$, and $\varphi(a)$ makes sense\).
\end{lemma}

\begin{proof} 
First assume that $a\in R\setminus\{0\}$.
Let $\fm$ denote the maximal ideal of $R$.
Let $v:Q(R)^\times\rightarrow \Bbb Z$ be the normalized discrete valuation 
associated with $R$.
Note that $\varphi^{-1}(R^\times)$, 
the image of $R^\times$ by $\varphi^{-1}$, 
is contained in $R^\times$.
So we have $\varphi(\fm)\subset \fm$.
It follows that $v(\varphi(a))\geq v(a)$.
Applying this to $\varphi^{-1}$, we have
$v(a)=v(\varphi^{-1}\varphi(a))\geq v(\varphi(a))$.
So $v(a)=v(\varphi(a))$, and $\varphi(a)\cdot a^{-1}\in R^\times$.

Next consider the case that $a^{-1}\in R\setminus\{0\}$.
Then $\varphi(a^{-1})\cdot a\in R^\times$.
Taking the inverse, $\varphi(a)\cdot a^{-1}\in R^\times$.
\end{proof}

\begin{lemma}\label{G-submod.thm}
Assume that $G(k)$ is dense in $G$.
Let $S$ be a Noetherian $G$-algebra, and $f\in Q(S)$.
If $Sf$ is a $G(k)$-submodule of $Q(S)$, then $f\in Q_G(S)$,
and $Sf$ is a $G$-submodule of $Q_G(S)$.
\end{lemma}

\begin{proof}
It is easy to see that 
\[
I:=S:_S Sf=\Ker(S\rightarrow \Hom_S(Sf,(Sf+S)/S))
\]
is a $G(k)$-submodule of
$S$.
By Lemma~\ref{dense.thm}, $I$ is a $G$-ideal of $S$.
If $f=b/a$ with $b\in S$ and $a\in S\nz$, then $a\in I$, and hence
$I$ contains a nonzerodivisor.
As $f\in S:_{Q(S)}I$, $f\in Q_G(S)$.
So $Sf$ is a $G(k)$-submodule of the $G$-module $Q_G(S)$, and
hence it is a $G$-submodule of $Q_G(S)$ by Lemma~\ref{dense.thm} again.
\end{proof}

\paragraph Let $V$ be a $G$-module and $f\in V$.
We say that $f$ is a semiinvariant under the action of $G$ if
the one-dimensional subspace $kf$ is a $G$-submodule of $V$.
Let $X(G)$ denote the set of isomorphism classes of characters
(i.e., one-dimensional $G$-modules) of $G$.
$X(G)$ is identified with the set of homomorphisms
$\Alggrp(G,\Bbb G_m)$ from $G$ to $\Bbb G_m$.
It is a subset of $\Mor(G,\Bbb A^1\setminus\{0\})=k[G]^\times$.
So we sometimes consider that $X(G)\subset k[G]^\times$.
Note that $X(G)$ is then identified with the group of group-like
elements of $k[G]$.

For $\chi\in X(G)$, define
\begin{multline*}
V^\chi:=\{f\in V\mid \text{$f$ is a semiinvariant, and $kf\cong \chi$}\}
\cup\{0\}
\\
=\{f\in V\mid \omega(f)=f\otimes \chi\},
\end{multline*}
where $\omega$ is the coaction of $V$.
Then $V^\chi\cong \Hom_G(\chi,V)$ is a subspace of $V$.
For a $G$-algebra $A$, $A_G:=\bigoplus_{\chi\in X(G)}A^\chi$ is a 
$X(G)$-graded $A^G$-algebra.

\paragraph\label{quasi-S_1.par}
A scheme $X$ is said to be quasi-$(S_1)$, if $X$ is the
union of finitely many irreducible closed subsets, and for any two
dense open subsets $U$ and $V$ of $X$ such that $V\subset U$, the
restriction map $\Gamma(U,\O_X)\rightarrow\Gamma(V,\O_X)$ is injective.
An open subscheme of a quasi-$(S_1)$ scheme is again quasi-$(S_1)$.
The inductive limit $\indlim \Gamma(U,\O_X)$, where $U$ runs through the
all dense open subsets of $X$, is denoted by $\R(X)$, and is called the
function ring of $X$.
For $f\in \R(X)$, there is a unique largest dense open subset $U(f)$ 
(resp.\ $U^*(f)$) of
$X$ such that $f\in \Gamma(U(f),\O_X)$ 
(resp.\ $f\in \Gamma(U(f),\O_X)^\times$).
$U(f)$ is called the domain of definition of $f$.
Note that $f:U(f)\rightarrow \Bbb A^1_{\Bbb Z}$ is a morphism, and
$U^*(f)=f^{-1}(\Bbb A^1_{\Bbb Z}\setminus 0)$,
where $0=\Spec \Bbb Z\subset \Bbb A^1_{\Bbb Z}$ is the origin.

\paragraph Let $S$ be a commutative ring.
Then $\Spec S$ is quasi-$(S_1)$ if and only if 
$S$ has only finitely many minimal primes, and 
$S\nz=S\setminus\bigcup_{
P\in \Min S}P$, where $\Min S$ denotes the set of minimal primes of $S$.
In this case, $\R(\Spec S)=Q(S)$.
If $X$ is a Noetherian scheme, then $X$ is quasi-$(S_1)$ if and only if
$X$ satisfies Serre's $(S_1)$-condition.

\paragraph A morphism $h:Y\rightarrow X$ of schemes is said to be
almost dominating if for any dense open subset $U$ of $X$, $h^{-1}(U)$ 
is dense in $Y$.
If $X$ and $Y$ are quasi-$(S_1)$ and $h:Y\rightarrow X$ is almost
dominating, the canonical map between the
function rings $h^*:\R(X)\rightarrow
\R(Y)$ is induced.
Let $X$ and $Y$ be the unions of finitely many irreducible closed subsets
and $h:Y\rightarrow X$ be a morphism.
Then
$h$ is almost dominating if and only if the generic point of an irreducible
component of $Y$ is mapped to the generic point of an irreducible component
of $X$ by $h$.
This condition is satisfied if $h$ is a flat morphism.
This condition is preserved by the base change by an open immersion.

\begin{lemma}\label{rat-lem.thm}
Let $h:Y\rightarrow X$ be a faithfully flat quasi-compact morphism 
between quasi-$(S_1)$ schemes.
Let $f\in \R(X)$, and $h^*(f)\in \Gamma(Y,\O_Y)$.
Then $f\in \Gamma(X,\O_X)$.
Thus identifying $\R(X)$ with a subring of $\R(Y)$ via the
injective map $h^*$,
we have $\R(X)\cap \Gamma(Y,\O_Y)=\Gamma(X,\O_X)$.
\end{lemma}

\begin{proof} 
Assume the contrary.
Then $U(f)\neq X$.
Taking an affine neighborhood $V$ of $x\in X\setminus U(f)$.
Then $h:h^{-1}(V)\rightarrow V$ is a faithfully flat morphism between
quasi-$(S_1)$ schemes, $f$ can be viewed as an element of 
$\R(V)$ (because $U(f)\cap V$ is dense in $V$), 
$f\notin \Gamma(V,\O_V)$, but $h^*(f)\in\Gamma(h^{-1}(V),\O_{h^{-1}
(V)})$.
Thus this is another counterexample to the lemma where $V$ is affine.
Replacing $X$ by $V$, we may assume
that $X=\Spec A$ is affine, to get a contradiction.

Take a faithfully 
flat morphism $\Spec B\rightarrow Y$.
Such a map exists, since $Y$ is quasi-compact.
Then $f\in B$.
In $Q(A)\otimes_A(B\otimes_A B)$, $1\otimes(f\otimes 1)$ and $1\otimes
(1\otimes f)$ agree, because $f\in Q(A)$.
As $A\rightarrow Q(A)$ is injective and $B\otimes_A B$ is flat over $A$, 
$B\otimes_A B\rightarrow Q(A)\otimes_A (B\otimes_A B)$ is injective, and
hence $f\otimes 1=1\otimes f$ in $B\otimes_A B$.
As
\[
0\rightarrow A\rightarrow B\xrightarrow{\alpha} B\otimes_A B
\]
is exact with $\alpha(b)=1\otimes b-b\otimes 1$, $f\in A$.
This is a contradiction.
\end{proof}

\begin{corollary}\label{rat-func.thm}
Let $h:Y\rightarrow X$ be a faithfully flat quasi-compact morphism between
quasi-$(S_1)$ schemes which is an open map, 
and $f\in \R(X)$.
Then $U(h^*(f))=h^{-1}(U(f))$, and $U^*(h^*(f))=h^{-1}(U^*(f))$.
\end{corollary}

\begin{proof} 
$U(h^*(f))\supset h^{-1}(U(f))$ is obvious.
On the other hand, $h:U(h^*(f))\rightarrow h(U(h^*(f)))$ is faithfully flat
quasi-compact, and $f\in\Gamma(h(U(h^*(f))),\O_X)$ by Lemma~\ref{rat-lem.thm}.
In other words, $h(U(h^*(f)))\subset U(f)$.
The first assertion follows.

Note that 
$U^*(h^*(f))$ is the inverse image of $\Bbb A^1_{\Bbb Z}\setminus 0$ by the
morphism
\[
h^{-1}(U(f))\xrightarrow{h} U(f)\xrightarrow{f} \Bbb A^1_{\Bbb Z}.
\]
So it is $h^{-1}(U^*(f))$.
\end{proof}

\begin{lemma}\label{Q_F-inverse.thm}
Let $R$ be a commutative ring, and $F$ an affine flat $R$-group scheme
of finite type.
Let $\varphi:S\rightarrow T$ be an injective $F$-algebra map between
$F$-algebra domains.
Then $Q(\varphi)^{-1}(Q_F(T))=Q_F(S)$.
\end{lemma}

\begin{proof}
We may assume that $S\subset T$ and $\varphi$ is the inclusion map.
So we may consider that $Q(S)$ is a subfield of $Q(T)$.

Note that $Q(T)\otimes_R R[F]\cap Q(S\otimes_R R[F])=Q(S)\otimes_R R[F]$
in $Q(T\otimes_R R[F])$.
Indeed, letting $k=Q(S)$, $X=\Spec(k\otimes_R R[F])$, 
$Y=\Spec(Q(T)\otimes_R R[F])$, we have that
\begin{multline*}
Q(T)\otimes_R R[F]\cap Q(S\otimes_R R[F])=\Gamma(Y,\O_Y)\cap 
\R(X)\\
=\Gamma(X,\O_X)=Q(S)\otimes_R
R[F]
\end{multline*}
by Lemma~\ref{rat-lem.thm} (recall that $X$ and $Y$ are affine algebraic
group-schemes over some fields, and hence are local complete intersections
\cite[(31.14)]{ETI}, and in particular, $(S_1)$).

It follows easily that
\[
Q(T)\otimes_R R\cap Q(S\otimes_R R[F])=
Q(T)\otimes_R R\cap Q(S)\otimes_R R[F]=Q(S)\otimes_R R.
\]

Now let $\iota_T:T\rightarrow T\otimes_R R[F]$ be the map given by
$\iota_T(t)=t\otimes 1$, and $\iota_T': Q(T)\rightarrow Q(T\otimes_R R[F])$ 
be the induced map.
Let $\gamma_T:T\otimes_R R[F]\rightarrow T\otimes_R R[F]$ be the map given
by $\gamma_T(t\otimes a)=\omega_T(t)(1\otimes a)$.
It corresponds to the isomorphism $\Spec T\times F\rightarrow \Spec T\times F$
given by $(t,f)\mapsto (f^{-1}t,f)$,
and hence $\gamma_T$ is an isomorphism.
Let $\gamma'_T:Q(T\otimes_R R[F])\rightarrow Q(T\otimes_R R[F])$ be the
induced map.
Note that $\gamma_T\iota_T=\omega_T$ and $\gamma'_T\iota'_T=\omega'_T$.
Note also that $\gamma'_T$ is an isomorphism, and
$\gamma'_T$ maps $Q(S\otimes_R R[F])$ bijectively onto
$Q(S\otimes_R R[F])$.
Now we have
\begin{multline*}
(\omega'_T)^{-1}(Q(S\otimes_R R[F]))=
(\iota'_T)^{-1}(Q(S\otimes_R R[F]))\\
=
Q(T)\cap Q(S\otimes_R R[F])=Q(S).
\end{multline*}

Note that $(\omega'_T)|_{Q(S)}:Q(S)\rightarrow Q(S\otimes_R R[F])$ is
nothing but $\omega'_S$, induced by $\omega_S$.
Now 
\begin{multline*}
Q_F(T)\cap Q(S)=(\omega'_T)^{-1}(Q(T)\otimes_R R[F])\cap
(\omega'_T)^{-1}(Q(S\otimes_R R[F]))\\
=(\omega'_T)^{-1}(Q(T)\otimes_R R[F]\cap Q(S\otimes_R R[F]))
=(\omega'_T)^{-1}(Q(S)\otimes_R R[F])\\
=
(\omega'_S)^{-1}(Q(S)\otimes_R R[F])=
Q_F(S)
\end{multline*}
by Lemma~\ref{omega-inverse.thm}.
This is what we wanted to prove.
\end{proof}

\paragraph\label{torus.par}
We give an example of $Q_F(S)$.
Let $R=\Bbb Z$, and $F=\Bbb G_m^n$, the split $n$-torus over $\Bbb Z$.
Let $S$ be an $F$-algebra domain.
In other words, $S=\bigoplus_{\lambda\in \Bbb Z^n}S_\lambda$ is a
$\Bbb Z^n$-graded domain, see \cite[(II.1.2.1)]{Hashimoto}.
Then $Q_F(S)$ is the localization $S_{\Gamma(S)}$ by the set 
of nonzero homogeneous elements $\Gamma(S)$ of $S$.

We prove this fact.
Note that $Q_F(S)=\indlim Q_F(T)$ by Lemma~\ref{Q_F-inverse.thm},
where the inductive limit is taken over $F$-subalgebras $T$ of
$S$ of finite type over $\Bbb Z$.
As we have $S_{\Gamma(S)}=\indlim T_{\Gamma(T)}$, replacing $S$ by $T$, 
we may assume that $S$ is of finite type over $\Bbb Z$.

So by Lemma~\ref{colon.thm}, $Q_F(S)=\bigcup_I S:_{Q(S)}I$, where 
the union is taken over all the nonzero homogeneous ideals $I$ of $S$.
If $\alpha\in Q_F(S)$, then $\alpha\in S:_{Q(S)}I$ for some $I$.
Taking a nonzero homogeneous element $s\in I$, $\alpha\in (1/s)S
\subset S_{\Gamma(S)}$.
Conversely, if $\alpha=a/s\in S_{\Gamma(S)}$ with $a\in S$ and
$s\in\Gamma(S)$, then $\alpha\in S:_{Q(S)}s\subset Q_F(S)$.
So $Q_F(S)=S_{\Gamma(S)}$, as desired.

\begin{lemma}\label{int-cl.thm}
Let $R$ be a commutative ring, $F$ a smooth affine $R$-group scheme
of finite type.
Let $A$ and $B$ be $F$-algebras, and $\varphi:A\rightarrow B$ an
$F$-algebra map.
Assume that $A$ is Noetherian.
Let $C$ be the integral closure of $\varphi(A)$ in $B$.
Then $C$ is an $F$-subalgebra of $B$.
\end{lemma}

\begin{proof} 
We may assume that $A\subset B$, and $\varphi$ is the inclusion.
Let 
$\omega:B\rightarrow B\otimes_R R[F]$ be the coaction.
Note that the integral closure of $A\otimes_R R[F]$ in $B\otimes_R R[F]$ is
$C\otimes_R R[F]$ by \cite[(6.14.4)]{EGA-IV}.
As $\omega(C)$ is integral over $A\otimes R[F]$, $\omega(C)\subset
C\otimes R[F]$.
This is what we wanted to prove.
\end{proof}

\begin{lemma}[cf.~Rosenlicht \cite{Rosenlicht2}]\label{Rosenlicht.thm}
Let $X$ be a reduced $k$-scheme of finite type.
Then there is a short exact sequence of the form
\[
1\rightarrow K^\times\xrightarrow{\iota}\Gamma(X,\O_X)^\times
\rightarrow \Bbb Z^r\rightarrow 0,
\]
where $K$ is the integral closure of $k$ in $\Gamma(X,\O_X)$,
and $\iota$ is the inclusion.
\end{lemma}

\begin{proof}
If $f:Y\rightarrow X$ is a dominating $k$-morphism with $Y$ being
a reduced $k$-scheme of finite type, and the lemma is true for $Y$,
then the lemma is true for $X$.
Thus we may assume that $X$ is affine and normal.
If the lemma is true for each connected component of $X$, then
the lemma is true for $X$.
So we may assume that $X$ is an affine normal variety.
Replacing $k$ by $K$, we may assume that $K=k$.
Let $X\hookrightarrow \bar X$ be an open immersion $k$-morphism such
that $\bar X$ is a projective normal variety.
Let $V_1,\ldots,V_n$ be the codimension one subvarieties of $\bar X$,
not intersecting $X$.
Let $v_1,\ldots,v_n$ be the corresponding normalized discrete valuations
of $k(X)$.
Then
\[
1\rightarrow k^\times=\Gamma(\bar X,\O_{\bar X})^\times
\xrightarrow{\iota}\Gamma(X,\O_X)^\times
\xrightarrow{v} \Bbb Z^n
\]
is exact, where $v(f)=(v_1(f),\ldots,v_n(f))$.
The assertion follows.
\end{proof}

\begin{lemma}[cf.~{\cite[(3.11)]{Kamke}}, {\cite[Theorem~1]{Popov2}}]
\label{principal.thm}
Let $G$ be connected.
Let $S$ be a $G$-algebra domain which is finitely generated over $k$.
Let $K$ be the integral closure of $k$ in $Q(S)$.
Assume that $X(G)\rightarrow X(K\otimes_k G)$ is surjective.
Then for $f\in Q(S)$, the following are equivalent.
\begin{description}
\item[(a)] $f\in Q_G(S)$, and $f$ is a semiinvariant of $Q_G(S)$.
\item[(b)] $U^*(f)$ is a $G$-stable open subset of $\Spec S$.
\item[(c)] $Sf\subset Q_G(S)$ is a $G$-submodule.
\end{description}
In particular, any unit of $S$ is a homogeneous unit of $S_G$.
\end{lemma}

\begin{proof} 
We prove the equivalence.
We may assume that $f\neq 0$.
Set $X:=\Spec S$.

{\bf (a)$\Rightarrow$(b)}.
Let $\omega':Q(S)\rightarrow Q(S\otimes k[G])$ be the map induced by
the coaction.
Then as $f$ is a semiinvariant, $\omega'(f)=f\otimes \gamma$ for
some group-like element $\gamma$ of $k[G]$.
Note that $\gamma$ is a unit of $k[G]$.
Let $\alpha:X\times G\rightarrow X$ be the morphism
given by $\alpha(x,g)=g^{-1}x$.
Note that $\omega'=\alpha^*:Q(S)\rightarrow Q(S\otimes k[G])$.
Now by Corollary~\ref{rat-func.thm},
\[
\alpha^{-1}(U^*(f))=U^*(\omega'(f))=U^*(f\otimes\gamma)=
U^*(f\otimes 1)=p_1^{-1}(U^*(f)),
\]
where $p_1:X\times G\rightarrow X$ is the first projection.
So the action $G\times U^*(f)\rightarrow X$ factors through $U^*(f)$,
and $U^*(f)$ is $G$-stable.

{\bf (b)$\Rightarrow$(a)}.
Let $S'$ be the normalization of $S$, and $\pi:\Spec S'\rightarrow X=\Spec S$
the associated map.
$\pi^{-1}(U^*(f))$ is a $G$-stable open subset of $\Spec S'$.
As $G$ is geometrically integral, each codimension one subvariety of
$\Spec S'$ contained in $\Spec S'\setminus\pi^{-1}(U^*(f))$ is $G$-stable.
As $U^*(\pi^*(f))\supset \pi^{-1}(U^*(f))$ and $\Spec S'\setminus
U^*(\pi^*(f))$ is the union of some
codimension-one subvarieties, $U^*(\pi^*(f))$ is
also $G$-stable.
So we may assume that $S$ is normal.
Then we have $K\subset S$.
Being the integral closure of $k$ in $S$, $K$ is a $G$-subalgebra of
$S$ by Lemma~\ref{int-cl.thm}.
Then by Lemma~\ref{connected-trivial.thm}, $K$ is $G$-trivial.

Assume that $Kf$ is a $K\otimes_k G$-submodule of $Q_G(S)$.
Then there is a one-dimensional $k$-subspace
$kcf$ of $Kf$ which is a $G$-submodule of $Kf$ for some $c\in K$.
But since $c$ is $G$-invariant, $kf$ is also a $G$-submodule of $Kf$, and
hence $f$ is semiinvariant.
Thus replacing $k$ by $K$ and $G$ by $K\otimes G$, we may assume that $k=K$.

By 
\cite[(6.14.2)]{EGA-IV} and \cite[(6.14.4)]{EGA-IV}, 
$k\sep\otimes_k S$ is a normal domain, and $k\sep$ is integrally closed
in $k\sep\otimes_k S$.
See also \cite{Cartan}.
If $1\otimes f\in Q_{k\sep\otimes G}(k\sep\otimes S)=Q_G(k\sep\otimes S)$, 
then
$f\in Q(S)\cap Q_G(k\sep\otimes S)=Q_G(S)$ by Lemma~\ref{Q_F-inverse.thm}.
Now $f$ is a semiinvariant if and only if 
$kf\subset Q_G(S)$ is a $G$-submodule
if and only if 
$k\sep\otimes kf \subset k\sep\otimes Q_G(S)$ is a $G$-submodule
if and only if 
$k\sep(1\otimes f)\subset Q_{k\sep\otimes G}(k\sep\otimes S)$ is
a $k\sep\otimes G$-submodule
if and only if $1\otimes f$ is a semiinvariant under the action of 
$k\sep\otimes G$.
On the other hand, the canonical map $\rho:\Spec(k\sep\otimes S)
\rightarrow \Spec S$ is a $G$-morphism, and 
$\rho^{-1}(U^*(f))=U^*(1\otimes f)$ by Corollary~\ref{rat-func.thm}.
As we assume that $U^*(f)$ is $G$-stable, 
$U^*(1\otimes f)$ is $G$-stable, and hence it is also 
$k\sep\otimes G$-stable.
So replacing $k$ by $k\sep$, $S$ by $k\sep\otimes_k S$, and $G$ by 
$k\sep\otimes_k G$, we may assume that $k=K=k\sep$.

Let $V$ be a finite dimensional $G$-submodule of $S$ which generates $S$
as a $k$-algebra.
Then
\[
\Spec S=X\hookrightarrow \Spec \Sym V\hookrightarrow \Proj \Sym (V\oplus k)
\]
is a sequence of immersions which are $G$-morphisms.
Let $Z$ denote the
normalization of the closure of the image of the composite of these morphisms.
Then $Z$ is a $k$-projective normal $G$-variety, and $X$ is its $G$-stable
open subset, as can be proved easily using Lemma~\ref{pre-normalization.thm}.

We want to prove that $kf$ is a $G$-submodule of $S$.
This is equivalent to say that $kf$ is a $G(k)$-submodule of $S$ by
Lemma~\ref{dense.thm}, since $G(k)$ is dense in $G$ by \cite[(AG13.3)]{Borel}.
So it suffices to show that for any $g\in G(k)$,
\begin{equation}\label{dvr.eq}
g(f)\cdot f^{-1}\in k=\Gamma(Z,\O_Z)=
   \bigcap_{V\subset Z,\;\codim_ZV=1}\mathcal O_{Z,V}.
\end{equation}
So let $g\in G(k)$ and $V$ be a codimension one subvariety of $Z$.

If $V$ is $G$-stable, then $g$ induces an automorphism of the DVR
$\O_{Z,V}$.
Then $g(f)\cdot f^{-1}\in \O_{Z,V}^\times$ by Lemma~\ref{dvr.thm}.
If $V$ is not $G$-stable, then $V\cap U^*(f)\neq\emptyset$, and hence
$f\in \O_{Z,V}^\times$.
Since $U^*(g(f))=g(U^*(f))=U^*(f)$ by assumption, 
$g(f)\in\O_{Z,V}^\times$.
Thus $g(f)\cdot f^{-1}\in\O_{Z,V}$.
This is what we wanted to prove.

{\bf (a)$\Rightarrow$(c)} is trivial.

{\bf (c)$\Rightarrow$(a)}.
Discussing as in the proof of {\bf (b)$\Rightarrow$(a)}, we may
assume that $S$ is normal, and $k=K=k\sep$.
Then since $Sf$ is a $G(k)$-submodule of $Q_G(S)$, $g(Sf)=Sf$ for
$g\in G(k)$.
Hence for a height-one prime ideal $P$ of $S$, 
\[
(Sf)_P\neq S_P\iff
g((Sf)_P)\neq g(S_P)\iff
(Sf)_{g(P)}\neq S_{g(P)}.
\]
It follows that $J:=\bigcap_{P\in X^1(S),\;(Sf)_P\neq S_P}P$ 
is a $G$-stable ideal of $S$.
Hence $U^*(f)=\Spec S\setminus V(J)$ 
is $G$-stable.
By {\bf (b)$\Rightarrow$(a)}, which has already been proved, we have 
that $f$ is a semiinvariant.

We prove the last assertion.
Set $A:=S_G$.
Let $f\in S^\times$.
Then $Sf=S$, and hence $f$ is a homogeneous element of $A$
by the first assertion,
which has already been proved.
Similarly, $f^{-1}\in A$.
This shows that $f$ is a unit of $A$.
\end{proof}

\iffalse
\begin{remark}\label{disputable.thm}
We mention a well-known (as a folklore) and
persuading, but somewhat disputable proof of
the implication {\bf (c)$\Rightarrow$(a)} in Lemma~\ref{principal.thm}.
For simplicity, assume that $k$ is algebraically closed.
Define $\chi: G(k)\rightarrow S^\times$ by
$gf=\chi(g)f$.
It suffices to prove that $\Image\chi\subset k^\times$.
Let $\bar\chi(g)$ be the image of $\chi(g)$ in $S^\times/k^\times$.
Note that $S^\times/k^\times\cong \Bbb Z^r$ for some $r$ by
Lemma~\ref{Rosenlicht.thm}.
Now as {\em $\Bbb Z^r$ is discrete and $G$ is connected}, $\bar \chi:G
\rightarrow \Bbb Z^r$ must be constant, and
$\chi$ maps $G$ to $k^\times$, as desired.
However, the continuity of $\bar\chi$ is a nontrivial problem,
which is (usually) left unproved.
\end{remark}
\fi

\begin{lemma}\label{fg-sub.thm}
Let $S$ be a $G$-algebra domain, and $f\in Q_G(S)$.
If $Sf$ is a $G$-submodule of $Q_G(S)$, then there is a finitely generated
$G$-subalgebra $T$ of $S$ such that $Tf$ is a $(G,T)$-submodule of
$Q_G(T)$.
\end{lemma}

\begin{proof}
We can write $f=b/a$ with $a,b\in S$ and $a\neq 0$.
Let $V$ be a finite dimensional $G$-submodule of $Sf$ containing $f$.
Let $s_1f,\ldots,s_nf$ be a $k$-basis of $V$, where $s_1,\ldots,s_n\in S$.
Let $T$ be a finitely generated $G$-subalgebra of $S$ containing 
$a,b,s_1,\ldots,s_n$.
Then $V\subset Q(T)\cap Q_G(S)=Q_G(T)$ by Lemma~\ref{Q_F-inverse.thm}.
Hence $Tf=T\cdot V$ is a $(G,T)$-submodule of $Q_G(T)$.
\end{proof}

\begin{corollary}\label{connected-cor.thm}
Let $G$ be connected.
Let $S$ be a $G$-algebra domain.
Let $K$ be the integral closure of $k$ in $Q(S)$.
Assume that 
$X(G)\rightarrow X(K\otimes_k G)$ is
surjective.
If $Sf\subset Q_G(S)$ is a $G$-submodule, then $f$ is a semiinvariant.
\end{corollary}

\begin{proof}
By Lemma~\ref{fg-sub.thm}, there exists some finitely generated 
$G$-subalgebra $T$ of $S$ such that $Tf$ is a $(G,T)$-submodule of $Q_G(T)$.
The integral closure $L$ of $k$ in $Q(T)$ lies between $k$ and $K$.
So the inclusion $X(G)\rightarrow X(L\otimes_k G)$ is bijective.
By Lemma~\ref{principal.thm}, $f$ is a semiinvariant.
\end{proof}

\begin{corollary}\label{principal-cor.thm}
Let $G$ be connected.
Let $S$ be a $G$-algebra domain.
Let $K$ be the integral closure of $k$ in $Q(S)$.
Assume that
$X(K\otimes_k G)$ is trivial.
Let $f\in Q(S)$, and assume one of the following:
\begin{description}
\item[(b)] $S$ is finitely generated and $U^*(f)$ is $G$-stable; or
\item[(c)] $Sf\subset Q_G(S)$ is a $G$-submodule.
\end{description}
Then $f\in Q(S)^G$.
In particular, $S^\times=(S^G)^\times$.
\qed
\end{corollary}

\begin{lemma}\label{pop-lemma.thm}
Let $S$ be a $G$-algebra domain such that $S^\times\subset S^G$.
Let $K$ be the integral closure of $k$ in $Q(S)$.
Let $G(K)$ be dense in $K\otimes_k G$.
Assume that $X(G)\rightarrow X(K\otimes_k G)$ is surjective.
If $f\in Q_G(S)$ and $Sf$ is a $G$-submodule of $Q_G(S)$, then
$f$ is a semiinvariant.
If, moreover, $X(G)$ is trivial, then $f\in S^G$.
\end{lemma}

\begin{proof}
We may assume that $f\neq 0$.
Note that $K^\times\subset S^\times\subset S^G$, and hence $K\subset S^G$.
Assume that $Kf$ is a $K\otimes_k G$-submodule of $Q_{K\otimes_k G}(S)$.
Then there is a one-dimensional $G$-stable $k$-subspace $kcf$ of $Kf$ with
$c\in K^\times$.
As $c\in K\subset S^G$, $kf$ is also $G$-stable, and hence $f$
is semiinvariant.
As we have $f\in Q_G(S)=Q_{K\otimes_k G}(S)$ and $Sf$ is a 
$K\otimes_k G$-submodule of $Q_G(S)$, replacing $k$ by $K$,
we may assume that $k=K$.

By Lemma~\ref{fg-sub.thm}, there exists some finitely generated $G$-subalgebra
$T$ of $S$ such that $Tf$ is a $(G,T)$-submodule of $Q_G(T)$.
Note that $T^\times\subset S^\times\cap T\subset S^G\cap T=T^G$.
Replacing $S$ by $T$, we may assume that $S$ if finitely generated.

We define $\chi:G(k)\rightarrow S^\times$ by $gf=\chi(g)f$ for $g\in G(k)$.
As $S^\times\subset S^G$, $\chi$ is a homomorphism of groups.
Let $\bar \chi(g)$ be the image of $\chi(g)$ in $S^\times/k^\times$.
Note also that $f$ is semiinvariant under the action of $G^\circ$ by
Lemma~\ref{principal.thm}.
So $\chi(G^\circ(k))\subset k^\times$.
So $\bar\chi$ induces a group homomorphism $\chi':G(k)/G^\circ(k)
\rightarrow S^\times/k^\times$.
As $G(k)/G^\circ(k)$ is finite and $S^\times/k^\times\cong \Bbb Z^r$ for
some $r$ by 
Lemma~\ref{Rosenlicht.thm},
$\chi'$ is trivial.
This shows that $\chi(G(k))\subset k^\times$.
So $kf$ is a $G(k)$-submodule of $Q_G(S)$.
By Lemma~\ref{dense.thm}, $f$ is a semiinvariant.
The last assertion is trivial.
\end{proof}

\begin{lemma}\label{unit-trivial.thm}
Let $G(k)$ be dense in $G$.
Let $S$ be a $G$-algebra domain such that $S^\times=k^\times$.
Let $f\in Q(S)^\times$.
If $Sf$ is a $G(k)$-submodule of $Q(S)$, then $f\in Q_G(S)$, and
$f$ is a semiinvariant.
If, moreover, $X(G)$ is trivial, then $f\in S^G$.
\end{lemma}

\begin{proof}
We show that $kf$ is a $G(k)$-submodule of $Q(S)$.
Let $g\in G(k)$.
Then since $Sf=g(Sf)=S(gf)$.
It follows that $gf=uf$ for some $u\in S^\times=k^\times$.

Now write $f=a/b$ with $a,b\in S$, $b\neq 0$.
Let $T$ be a finitely generated $G$-subalgebra of $S$ containing $a$ and $b$.
Then $kf$ is a $G(k)$-submodule of $Q(T)$.
In particular, $Tf$ is a $G(k)$-submodule of $Q(T)$, and hence 
$f\in Q_G(T)\subset Q_G(S)$ by Lemma~\ref{G-submod.thm}.
Since $kf$ is a $G(k)$-submodule of $Q_G(T)$, $f$ is a semiinvariant by
Lemma~\ref{dense.thm}.
The last assertion is trivial.
\end{proof}

\begin{remark}\label{character.rem}
Let $k$ be an algebraically closed field.
\begin{description}
\item[(i)] If $N$ is a normal closed subgroup of $G$ with $X(N)=\{e\}$, 
then the restriction $X(G/N)\rightarrow X(G)$ is an isomorphism.
\item[(ii)] If $N$ is a unipotent group, then $X(N)$ is trivial.
In particular, $X(G)$ is trivial if and only if $X(G/R_u)$ is trivial,
where $R_u$ is the unipotent radical of $G$.
Note that the identity component of $G/R_u$ is reductive.
\item[(iii)] If $G^\circ$ is reductive, then $H=G/[G,G]$ is an abelian
group such that $H^\circ$ is a torus.
Note that $X(H)\rightarrow X(G)$ is an isomorphism.
It is easy to see that $H\cong H^\circ\times M$ for some abelian finite group
$M$.
Thus $X(G)$ is trivial if and only if $M=G/[G,G]$ is finite, and 
the order of $M$ is a power of $p$, where $p=\max(\charac(k),1)$.
\item[(iv)] In particular, if $G^\circ/R_u$ is semisimple and the order of
$G/G^\circ$ is a power of $p$, then $X(G)$ is trivial.
\item[(v)] If $G$ is connected, then 
$G/(R_u\cdot [G,G])$ is a torus, and $X(G)\cong X(G/R_u\cdot [G,G])$
is a finitely generated free abelian group.
$X(G)$ is trivial if and only if
$G/R_u$ is semisimple if and only if the radical of $G$ is unipotent.
\item[(vi)] If $G$ is connected, $X([G,G])$ is trivial by {\bf (v)}.
\end{description}
\end{remark}

\begin{lemma}\label{principal-invariant.thm}
Let $R$ be a commutative ring, 
$F$ an affine flat $R$-group scheme, and
$S$ an $F$-algebra, $A=S^F$, and $f\in A\cap S\nz$.
Then $Af=Sf\cap A$.
\end{lemma}

\begin{proof} 
The multiplication by $f$ induces an $(F,S)$-isomorphism 
$f:S\rightarrow Sf$.
Taking the $F$-invariance, we have that the multiplication 
$f:A\rightarrow (Sf)^F=Sf\cap A$ is an isomorphism.
\end{proof}

\begin{lemma}\label{H-UFD.thm}
Let $R$ be a Noetherian commutative ring, and 
\begin{equation}\label{exact.eq}
1\rightarrow N\rightarrow F\rightarrow T\rightarrow 1
\end{equation}
be an exact sequence of affine $R$-group schemes flat of finite type.
Assume that $T\cong \Bbb G_m^n$ is a split torus.
Let $S$ be a Noetherian $F$-algebra which is an $F$-domain.
That is, $\Spec S$ is $F$-integral, see {\em\cite[(4.12)]{HM}}.
Assume that for any height one $F$-prime $F$-ideal $P$ such that
$P\cap S^N$ is a minimal prime of a nonzero principal ideal, 
$P\cap S^N$ is principal 
\(for the definition of an $F$-prime $F$-ideal, see {\rm\cite[(4.12)]{HM}}\).
Then $S^N$ is a UFD.
\end{lemma}

\begin{proof}
Set $A:=S^N$.
Note that $A$ is a $T$-algebra so that $A$ is an $F$-subalgebra of 
$S$ in a natural way by Lemma~\ref{exact.thm}.
As $0$ of $S$ is an $F$-prime $F$-ideal, $0=0\cap A$ of $A$ is
also an $F$-prime $F$-ideal by \cite[(4.14)]{HM}.
So $0$ of $A$ is a $T$-prime $T$-ideal.
As $T$ is $R$-smooth with connected fibers, $0$ is a prime ideal of
$A$ by \cite[(6.25)]{HM}.
That is, $A$ is an integral domain.

Note that $A$ is a $\Bbb Z^n$-graded $R$-algebra.
Assume that $A$ is not a UFD.
Then, by Lemma~\ref{graded-ufd2.thm},
there is a nonzero homogeneous element of $A$ which is not a unit or
a product of prime elements.
As $S$ is Noetherian, the set
\begin{multline*}
\{Sa\mid a\in A\setminus(A^\times\cup\{0\}),\;
\text{ $a$ is homogeneous,} \\
\text{and not a product of prime elements in $A$}\}
\end{multline*}
has a maximal element $Sa$ with $a\in A\setminus(A^\times\cup\{0\})$.
Note that $Sa\neq S$, otherwise $a\in S^\times\cap A=A^\times$.
If $P$ is a height zero associated prime of $Sa$, then
$0=P^*\supset (Sa)^*=Sa\ni a\neq 0$, and this is a contradiction.
So any minimal prime of $Sa$ is height one.
As the $F$-prime $0$ of $S$ does not have an embedded prime \cite[(6.2)]{HM},
$a$ is a nonzerodivisor in $S$.
In particular, $Sa\cap A=Aa$ by Lemma~\ref{principal-invariant.thm}.

Let 
\[
\sqrt[F]{Sa}=P_1\cap\cdots\cap P_s
\]
be a minimal $F$-prime decomposition.
Then
\[
\sqrt{Aa}=(P_1\cap A)\cap\cdots\cap(P_s\cap A).
\]
Indeed, each $P_i\cap A$ is a prime ideal containing $Aa$ by 
\cite[(4.14)]{HM} and \cite[(6.25)]{HM}.
On the other hand, $\sqrt[F]{Sa}\cap A\subset \sqrt{Sa}\cap A=\sqrt{Aa}$.
So there is some $i$ such that $P_i\cap A$ is a minimal prime of $Aa$.
Replacing $P_i$ by a smaller one if necessary, we may assume that $P_i$
is minimal so that $\height P_i=1$.
Now $P_i\cap A=Ab$ for some $b\in A\setminus(A^\times\cup\{0\})$ by
assumption.
So $a$ is divisible by a prime element $b$.
As $a$ is homogeneous, $b$ and $a/b$ are homogeneous elements of $A$.
Note that $Sa\subsetneq S(a/b)$ (otherwise $b\in S^\times\cap A=A^\times$).
The choice of $a$ shows that either $a/b$ is in $A^\times$, or 
$a/b$ is a product of prime elements.
So $a=b\cdot(a/b)$ is a product of prime elements.
This is a contradiction, and $A$ is a UFD.
\end{proof}

\paragraph Let $G$ be connected.
Then by (v) of Remark~\ref{character.rem}, $X(G)\cong \Bbb Z^n$ for
some $n$.
Note that 
%%% revised 2010/12/2
$kX(G)$ is a $k$-subbialgebra of $k[G]$.
So the inclusion $kX(G)\hookrightarrow k[G]$ induces a surjective
homomorphism $\varphi:
G\rightarrow T$, where $T:=\Spec kX(G)\cong \Bbb G_m^n$.
%%%
Let $N:=\Ker \varphi$.
Then it is easy to see that
\[
1\rightarrow N\rightarrow G\rightarrow T\rightarrow 1
\]
is a short exact sequence of $k$-groups.

For a $G$-module $M$, 
\[
M_G:=\bigoplus_{\chi\in X(G)}\Hom_G(\chi,M)
\cong (M\otimes k[T])^G\cong (M\otimes k[G]^N)^G\cong M^N,
\]
where $G$ acts on $k[T]$ and $k[G]^N$ left regularly, and $N$
acts on $k[G]$ right regularly.
The last isomorphism is given by $\sum_i m_i\otimes f_i\mapsto
\sum_i f_i(e)m_i$, where $e$ is the unit element of $G$.
The inverse $M^N\rightarrow (M\otimes k[G]^N)^G$ is given by 
the restriction of the coaction $\omega:M\rightarrow M\otimes k[G]$.
In particular, for a $G$-algebra $S$, we have $S_G=S^N$.
So if $S$ is a Krull domain (resp.\ integrally closed domain), then
so is $S_G=S\cap Q(S)^N$.
By Lemma~\ref{H-UFD.thm}, we immediately have the following.

\begin{corollary}\label{key-cor.thm}
Let $G$ be connected.
Let $S$ be a Noetherian $G$-algebra domain.
Assume that for any $G$-stable 
height one prime $P$ such that
$P\cap S_G$ is a minimal prime of a nonzero principal ideal, 
$P\cap S_G$ is principal.
Then $S_G$ is a UFD.
\qed
\end{corollary}

\begin{lemma}\label{symbolic-power.thm}
Let $R$ be a commutative ring, $F$ an affine flat $R$-group scheme of
finite type.
Let $S$ be a Noetherian $F$-algebra.
If $P$ is a prime ideal which is an $F$-ideal, then $P^{(n)}:=P^n S_P\cap S$ 
is an $F$-ideal for $n\geq 1$.
\end{lemma}

\begin{proof}
The only minimal prime of $P^n$ is $P$.
In particular, the only minimal $F$-prime of $P^n$ is $P^*=P$.
Let $I$ be the $F$-primary component of $P^n$ corresponding to the
minimal $F$-prime $P$ (it is unique, see \cite[(5.17)]{HM}).
It has only one associated prime $P$ by \cite[(6.13)]{HM}.
So by \cite[(6.10)]{HM}, $I$ is the $P$-primary component $P^n$.
Thus $I=P^{(n)}$, and $P^{(n)}$ is an $F$-ideal.
\end{proof}

\begin{theorem}\label{main.thm}
Let $S$ be a finitely generated $G$-algebra which is a normal domain.
Assume that $G$ is connected.
Assume that $X(G)\rightarrow X(K\otimes_k G)$ is surjective, where
$K$ is the integral closure of $k$ in $S$.
Let $X^1_G(S)$ be the set of height one $G$-stable prime ideals of $S$.
Let $M(G)$ be the subgroup of the class group $\Cl(S)$ of $S$ generated
by the image of $X^1_G(S)$.
Let $\Gamma$ be a subset of $X^1_G(S)$ whose image in $M(G)$ generates
$M(G)$.
Set $A:=S_G$.
Assume that $Q_G(S)_G\subset Q(A)$.
Assume that if $P\in\Gamma$, then either the height of $P\cap A$ is not one
or $P\cap A$ is principal.
Then for any 
$G$-stable height one prime ideal $Q$ of $S$, either the height of 
$Q\cap A$ is not one or $Q\cap A$ is a principal ideal.
In particular, $A$ is a UFD.
If, moreover, $X(G)$ is trivial, then $S^G=A$ is a UFD.
\end{theorem}

\begin{proof}
Note that $\{[P]\mid P\in \Gamma\}$ generates $\I(M(G))\subset \Cl'(S)$,
see for the notation, (\ref{class-group.par}).
So the class $[Q]$ of $Q$ in $\Cl'(S)$ is equal to 
$\sum_{i=1}^r n_i [P_i]$ for some $P_1,\ldots,P_r\in\Gamma$.
We may assume that $n_1,\ldots,n_s\geq 0$, and $n_{s+1},\ldots,n_r<0$.
Let $I=P_1^{(n_1)}\cap \cdots\cap P_s^{(n_s)}$, and
$J=P_{s+1}^{(-n_{s+1})}\cap \cdots\cap P_r^{(-n_r)}$.
By Lemma~\ref{symbolic-power.thm}, both $I$ and $J$ are $G$-stable.
Note that $[Q]=[I:_{Q(S)} J]$.
So there exists some $\alpha\in Q(S)$ such that $\alpha Q=I:_{Q(S)} J$.
Then $\alpha\in I:_{Q(S)}JQ\subset Q_G(S)$ by Corollary~\ref{colon-cor.thm}.
As $\alpha S=I:_{Q(S)}JQ$ is a $(G,S)$-submodule of $Q_G(S)$ by
Corollary~\ref{colon-cor.thm}, $\alpha\in Q_G(S)_G\subset Q(A)$ by
Lemma~\ref{principal.thm} and the assumption.
Hence
\begin{equation}\label{QcapA.eq}
\alpha(Q\cap A)=(I:_{Q(S)}J)\cap Q(A).
\end{equation}
Now assume that $Q\cap A$ is height one.
Then the left hand side is a divisorial fractional ideal of the Krull
domain $A$, and hence so is the right hand side.

Applying Lemma~\ref{divisorial.thm}, if
$\fp\in X^1(A)$ and $((I:_{Q(S)}J)\cap Q(A))_\fp\neq A_\fp$, then
$n_\fp(I:_{Q(S)}J)\neq 0$.
In particular, there exists some $P\in X^1(\fp)$ such that
$v_P(I:_{Q(S)}J)\neq 0$.
Such a $P$ must be one of $P_1,\ldots,P_r$, and lies in $\Gamma$.
This forces that $\fp$ is principal, by assumption.
This shows that $\alpha(Q\cap A)=(I:_{Q(S)}J)\cap Q(A)$ is principal.
So $Q\cap A$ is also principal, as desired.

We prove that $A$ is a UFD.
As $A$ is a Krull domain, $P$ is a height one prime if and only if it is a 
minimal prime ideal of a nonzero principal ideal.
By Lemma~\ref{key-cor.thm}, $A$ is a UFD.

The last assertion is trivial.
\end{proof}

\begin{theorem}\label{ufd.thm}
Let $G$ be connected.
Let $S$ be a $G$-algebra of finite type over $k$.
Assume that $S$ is an integral domain.
Assume that $X(G)\rightarrow X(K\otimes_k G)$ is surjective,
where $K$ is the integral closure of $k$ in $Q(S)$.
Set $A:=S_G$.
Assume that if $P$ is a 
$G$-stable height one prime ideal of $S$ such that 
%%%
if
%%%
$P\cap A$ is a
minimal prime of some nonzero principal ideal,
then $P$ is a principal ideal.
Then 
\begin{description}
\item[(i)] 
If $P$ is a $G$-stable height one prime ideal of $S$ such that
$P\cap A$ is a minimal prime of a nonzero principal ideal,
then $P=Sf$ for some homogeneous prime element $f$ of $A$.
\item[(ii)] $A$ is a UFD.
\item[(iii)] Any homogeneous 
prime element of $A$ is a prime element of $S$.
\item[(iv)] If, moreover, $X(G)$ is trivial, then $S^G=A$ is a UFD.
\end{description}
\end{theorem}

\begin{proof}
{\bf (i)}
By assumption, we can write $P=Sf$ with $f\in S$.
By Lemma~\ref{principal.thm}, $f$ is a homogeneous element of $A$.
As $Af=A\cap P$ by Lemma~\ref{principal-invariant.thm}, 
$f$ is a prime element of $A$.

{\bf (ii)} Let $P$ be a $G$-stable height one prime ideal such that
$P\cap A$ is a minimal prime ideal of a nonzero principal ideal.
Note that $P=Sf$ for some prime element $f\in A$ by {\bf (i)}.
Then $P\cap A=Af$ is principal.
Now by Corollary~\ref{key-cor.thm}, $A$ is a UFD.

{\bf (iii)} 
Let $a$ be  a homogeneous prime element of $A$.
Let $P$ be a minimal prime of $Sa$ such that $P\cap A=Aa$.
Then $P=Sb$ for some homogeneous 
prime element $b$ of $A$ by {\bf (i)}, and
$Ab=P\cap A=Aa$.
Hence $Sa=Sb=P$, and $a$ is a prime element of $S$.

{\bf (iv)} is trivial.
\end{proof}

\begin{corollary}[cf.~{\cite[p.~376]{Popov2}}]
\label{ufd-cor.thm}
Let $G$ be connected.
Let $S$ be a $G$-algebra of finite type over $k$.
Assume that $S$ is a UFD.
Assume that $X(G)\rightarrow X(K\otimes_k G)$ is 
surjective, where $K$ is the
integral closure of $k$ in $S$.
Then $A:=S_G$ is a UFD.
Any 
%%% added 2010/11/25
homogeneous
%%%
prime element of $A$ is a prime element of $S$.
If, moreover, $X(G)$ is trivial, then $S^G=A$ is a UFD.
\qed
\end{corollary}

We can prove Corollary~\ref{ufd-cor.thm} without assuming that
$S$ is finitely generated.

\begin{proposition}\label{ufd-cor2.thm}
Let $G$ be connected.
Let $S$ be a $G$-algebra.
Assume that $S$ is a UFD.
Assume also that $X(G)\rightarrow X(K\otimes_k G)$ is surjective, where
$K$ is the integral closure of $k$ in $S$.
Then $A:=S_G$ is a UFD.
Any homogeneous prime element of $A$ is a prime element of $S$.
\end{proposition}

\begin{proof}
Let $f\in A\setminus(A^\times\cup\{0\})$ be a homogeneous element.
Then $f\in S\setminus(S^\times\cup\{0\})$.
Let $f=f_1\cdots f_r$ $(r\geq 1)$ be a prime factorization of $f$ in $S$.
For each $i$, the scheme theoretic image $V(Sf_i)^*$ of the action 
$G\times V(Sf_i)\rightarrow \Spec S$
is integral by Lemma~\ref{star-integral.thm}, 
is contained in $V(Sf)$ (because $Sf$ is a $G$-stable ideal),
and contains $V(Sf_i)$.
But there is no integral closed subscheme $E$ of $\Spec S$ such that
$V(Sf_i)\subsetneq E\subset V(f)$.
So $V(Sf_i)^*=V(Sf_i)$.
In other words, $Sf_i$ is $G$-stable.
By Lemma~\ref{principal.thm}, $f_i\in A$.
By Lemma~\ref{principal-invariant.thm}, $Af_i=Sf_i\cap A$, and hence
$f_i$ is a prime element of $A$.
This shows that $f$ has a prime factorization in $A$, and $A$ is a UFD
by Lemma~\ref{graded-ufd2.thm}.
Moreover, if $f$ is not irreducible in $S$ and $r\geq 2$, then
$f$ is not irreducible in $A$.
This shows that any homogeneous
prime element of $A$ is a prime element of $S$.
\end{proof}

\begin{lemma}[cf.~{\cite[p.~376]{Popov2}}]\label{ufd-easier.thm}
Let $S$ be a $G$-algebra which is a UFD.
Assume that $G(K)$ is dense in $K\otimes_k G$,
where $K$ is the integral closure of $k$ in $S$.
Assume that $X(K\otimes_k G)$ is trivial.
Assume also that $S^\times\subset A=S^G$.
Then $A$ is a UFD.
\end{lemma}

\begin{proof}
As $K^\times\subset S^\times\subset S^G$, $K\subset S^G$.
So replacing $k$ by $K$, we may assume that $k=K$.

Let $f\in A\setminus(\{0\}\cup A^\times)$.
Let $f=f_1\cdots f_r$ be a prime factorization of $f$ in $S$.
Using induction on $r$, we prove that $f$ has a prime factorization in $A$.

As $G(k)$ leaves the ideal $Sf$ stable, $G(k)$ acts on the set of
prime ideals $\Gamma=\{(f_1),\ldots,(f_r)\}$.
Let $\Gamma_1$ be a $G(k)$-orbit, and let $\Gamma_1=\{(f_{i_1}),\ldots,
(f_{i_s})\}$, with $(f_{i_j})$ distinct.
Set $h=\prod_{j}f_{i_j}$.
Then $G(k)$ only permutes the elements of $\Gamma_1$, and 
the ideal $Sh$ is $G(k)$-invariant.
By Lemma~\ref{pop-lemma.thm}, $h\in A$.
If $a,b\in A$ and $ab\in hA$, then either $a\in f_{i_1}S$ or $b\in f_{i_1}S$.
If $a\in f_{i_1}S$, then $a\in hS\cap A=hA$.
If $b\in f_{i_1}S$, then $b\in hS\cap A=hA$.
As $h\notin A^\times=S^\times$, $h$ is a prime element of $A$.
Note that $f/h\in S\cap Q(A)=A$.
Using the induction assumption, $f/h$ is either a unit, or has a
prime factorization in $A$.
So $f=h(f/h)$ has a prime factorization in $A$, as desired.
\end{proof}

\begin{lemma}[cf.~{\cite[section~6]{Hochster}}, 
{\cite[Lemma~2]{Popov}}]
\label{ufd-easy.thm}
Let $S$ be a $G$-algebra which is a UFD.
Assume that $S^\times=k^\times$.
If $G(k)$ is dense in $G$ and $X(G)$ is trivial, then $S^G$ is a UFD.
\end{lemma}

\begin{proof}
Note that $k^\times\subset K^\times\subset S^\times=k^\times$.
So $K=k$.
The assertions follow immediately from Lemma~\ref{ufd-easier.thm}.
\end{proof}

\section{The Italian problem on invariant subrings}

The following is a refinement of \cite[Lemma~1]{Popov}.
See also \cite[(3.14)]{Kamke}.

\begin{proposition}\label{Kamke.thm}
Let $G$ be connected.
Let $S$ be a $G$-algebra which is a Krull domain.
Assume also that any $G$-stable height one prime ideal of $S$ is principal.
Moreover, assume that
$X(G)\rightarrow X(K\otimes_k G)$ is surjective,
where $K$ is the integral closure of $k$ in $S$.
Then $Q_G(S)_G=Q_T(A)$, where $T=\Spec k X(G)$.
If, moreover, $X(G)$ is trivial, then $Q(S)^G=Q(S^G)$. 
\end{proposition}

\begin{proof}
Let $a/b\in Q_G(S)_G\setminus\{0\}$ be a homogeneous element of
$Q_G(S)_G$ with $a,b\in S\setminus\{0\}$.
As $k\cdot(b/a)$ is a $G$-submodule of $Q_G(S)$, 
$S(b/a)$ is a $(G,S)$-submodule of $Q_G(S)$.
Hence $I:=Sb:_S Sa= S(b/a)\cap S$ is a $G$-ideal of $S$.
Clearly $I$ is divisorial.

Let $P_1,\ldots,P_s$ be the minimal primes of $I$.
Each $P_i$ is height one, and is $G$-stable by
Lemma~\ref{star-integral.thm}.

So $P_i=Sf_i$ for some $f_i\in S$ by assumption.
Then $f_i$ is a semiinvariant by Lemma~\ref{principal.thm}.
Now $f:=(f_1\cdots f_s)^n\in I\setminus\{0\}$ for sufficiently large $n$.
Then $h:=f(a/b)\in S$ is also a semiinvariant.
So $h\in A$.
Then $a/b=h/f\in Q_T(A)$ by (\ref{torus.par}).
This shows that $Q_G(S)_G\subset Q_T(A)$.

On the other hand, $A=S_G\subset Q_G(S)_G$.
Moreover, if $s$ is a semiinvariant and $\omega(s)=s\otimes \gamma$ for
a group-like element $\gamma$ of $k[G]$, then $\omega'(s^{-1})=s^{-1}
\otimes \gamma^{-1}$, and hence $s^{-1}\in Q_G(S)_G$.
This shows that $Q_T(A)\subset Q_G(S)_G$.
Hence $Q_G(S)_G=Q_T(A)$.

The last assertion is obvious.
\end{proof}

\begin{remark}
The heart of the argument above is in \cite[section~3]{Kamke}.
\end{remark}

\paragraph\label{invariant-field-var.par}
Let $X$ be an $(S_1)$ $k$-scheme of finite type on which $G$ acts.
Let $\Phi: G\times X\rightarrow X\times X$ be the morphism $\Phi(g,x)=(gx,x)$.
Let $g:=\dim G$, and $s=g-\min\{\dim G_x \mid x\in X\}$, where 
$G_x$ is the stablizer of $x$.
That is, $G_x=\Phi^{-1}(x,x)$.
It is a closed subgroup scheme of $G\times x$ over $x$.
Set $U:=\{x\in X\mid \dim G_x=g-s\}$.
Note that $U$ is a non-empty open subset of $X$.
Let $b:G\times X\rightarrow X$ be the map given by $b(g,x)=g^{-1}x$.
Let $p_2:G\times X\rightarrow X$ be the projection.
Both maps are flat, so 
the maps $b^*:\R(X)\rightarrow \R(G\times X)$ and
$p_2^*:\R(X)\rightarrow \R(G\times X)$ are induced.
We denote the kernel of $b^*-p_2^*$ by $\R(X)^G$.
If $X$ is a variety, then we write it as $k(X)^G$.
It is a subfield of $k(X)$ in this case.

If $X=\Spec S$ is affine, then $\R(X)^G=Q(S)^G$ by definition.
On the other hand, for $g\in G(k)$ and $f\in \R(X)$, $(gf)(x)=f(g^{-1}x)$ 
gives a definition of the action of $G(k)$ on $\R(X)$.

\begin{lemma}
In general, $\R(X)^G\subset \R(X)^{G(k)}$.
If $G(k)$ is dense in $G$, then $\R(X)^G=\R(X)^{G(k)}$.
\end{lemma}

\begin{proof}
Let $q: \Spec B\rightarrow X$ be a faithfully flat morphism of finite type
such that $B$ satisfies the $(S_1)$ condition.
Let $h$ be the composite
\[
Y:=G\times\Spec B\xrightarrow{1\times q}G\times X\xrightarrow{a}X,
\]
where $a$ is the action.
Note that $h$ is a faithfully flat $G$-morphism of finite type.
The diagram
\[
\xymatrix{
0 \ar[r] & \R(X)^G \ar[r] 
& \R(X) \ar[r]^-{b^*-p_2^*} \ar[d]^{h^*} 
& \R(G\times X) \ar[d]^{(1\times h)^*} \\
0 \ar[r] & \R(Y)^G \ar[r]
& \R(Y) \ar[r]^-{b^*-p_2^*}
& \R(G\times Y)
}
\]
is commutative with exact rows and injective vertical arrows.
So $\R(X)^G=\R(X)\cap \R(Y)^G$ in $\R(Y)$.
On the other hand, $h^*:\R(X)\rightarrow \R(Y)$ is a $G(k)$ homomorphism,
and $\R(X)^{G(k)}=\R(X)\cap 
%%% revised 2010/12/02
\R
%%%
(Y)^{G(k)}$.

As $Y$ is affine, $\R(Y)^G\subset \R(Y)^{G(k)}$ in general, and 
$\R(Y)^G=\R(Y)^{G(k)}$ if $G(k)$ is dense in $G$ by 
Lemma~\ref{field-abstract-grp.thm}.
The assertions follow immediately.
\end{proof}

\paragraph Let $k\rightarrow K$ be an algebraic extension of fields.
Let $X$ and $G$ be as in (\ref{invariant-field-var.par}).
Then the canonical isomorphism $K\otimes_k \R(X)\rightarrow \R(K\otimes_k X)$ 
induces an isomorphism $K\otimes_k \R(X)^G\cong \R(K\otimes_k X)^G$.

\begin{lemma}
Let $X$ be a $k$-variety on which $G$ acts.
Then $s=\tdeg_k k(X)-\tdeg_k k(X)^G$, where $\tdeg$ denotes the transcendence 
degree.
\end{lemma}

\begin{proof}
Replacing $G$ by $G^\circ$, we may assume that $G$ is connected.

Let $\rho:X'\rightarrow X$ be the normalization of $X$.
Note that $X'$ is a $G$-scheme in a unique way so that $\rho$ is a
$G$-morphism.
Let $x'$ be a closed point of $X'$, and $x:=\rho(x')$.
As $\rho^{-1}(x')$ is finite over $x$, 
$((G_x)\times_x x')^\circ\subset G_{x'}$, 
set theoretically.
On the other hand, if $gx'=x'$, then $gx=x$, and hence 
$G_{x'}\subset G_x\times_x x'$.
In particular, $\dim G_{x'}=\dim G_x$.
So replacing $X$ by $X'$, we may assume that $X$ is normal.
Let $K$ be the integral closure of $k$ in $\Gamma(X,\O_X)$.
Then $K$ is integrally closed in $k(X)$ by normality.
Note that $\Gamma(X,\O_X)$ is a $G$-algebra, and thus $K$ is a 
$G$-subalgebra of $\Gamma(X,\O_X)$.
So $G$ acts on $K$ trivially.
Replacing $k$ by $K$, we may assume that $k=K$.
Then taking a base change by $k\sep$, we may assume that $k=K=k\sep$.

The morphism $\Phi:G\times X\rightarrow X\times X$ induces
$\alpha:k(X)\otimes k(X)\rightarrow C$,
where $C=\{\psi\in k(G\times X)\mid\forall g\in G(k)\;(g\times X)\cap U(\psi)
\neq\emptyset\}$.
Note that $\alpha(f\otimes h)=a^*(f)(1\otimes h)$, where 
$a:G\times X\rightarrow X$ is the action.
As $a^*(f)=1\otimes f$ for $f\in k(X)^G$, $\alpha$ induces 
$\bar \alpha:k(X)\otimes_{k(X)^G}k(X)\rightarrow C$.
For any $g\in G(k)$, $g:C\rightarrow k(X)$ given by $(g\psi)(x)=\psi(g,x)$
is well-defined.
Note that $g(a^*(f))=g^{-1}f$.

We show that $\bar\alpha$ is injective.
Let $\sum_{j=1}^n f_j\otimes h_j\in\Ker\bar \alpha$ with 
$f_1,\ldots,f_n$ linearly independent over $k(X)^G$.
As $k(X)^G=k(X)^{G(k)}$, there exist some $g_1,\ldots,g_n\in G(k)$ such that
$\det(g_i f_j)\neq 0$ by Artin's lemma \cite[\S7, n${}^\circ$1]{Bourbaki}.
Then 
\[
0=g_i^{-1}\bar\alpha(\sum_j f_j \otimes h_j)=\sum_j g_i f_j\otimes h_j
\]
for $i=1,\ldots, n$.
This shows that $h_1,\ldots,h_n=0$, and hence $\bar\alpha$ is injective,
as desired.

Let $Z$ be the scheme theoretic image of $\Phi$.
Then by the last paragraph, $k(Z)\cong Q(k(X)\otimes_{k(X)^G}k(X))$.
Hence $\dim Z=2\dim X-\tdeg_k k(X)^G$.
On the other hand, $\Phi^{-1}\Phi(g,x)\cong g G_x $ for $g\in G$ and
$x\in X$.
Thus $\dim Z=\dim X+s$.
So $s=\dim X-\tdeg_k k(X)^G=\tdeg_k k(X)-\tdeg_k k(X)^G$, as desired.
\end{proof}

\paragraph Let $S$ be a finitely generated $k$-algebra domain.
Set $r_S=r=\tdeg_k Q(S)-\tdeg_k Q(S^G)$, and $s_S=s
=\tdeg_k Q(S)-\tdeg_k Q(S)^G$,
as before.
Obviously, $r\geq s$ in general.

\begin{lemma}\label{Italian.thm}
If $S$ is normal, then $Q(S^G)=Q(S)^G$ if and only if $r=s$.
\end{lemma}

\begin{proof} 
To say that $r=s$ is the same as 
$\tdeg_k Q(S^G)=\tdeg_k Q(S)^G$.
Or equivalently, the extension $Q(S^G)\rightarrow Q(S)^G$ is algebraic.
So the \lq only if' part is obvious.

We prove the \lq if' part.
Set $A=S^G$.
Let $\alpha\in Q(S)^G$.
By assumption, $\alpha$ is integral over $A[1/a]$, for some $a\in A\setminus
\{0\}$.
As $\alpha$ is integral over $S[1/a]$, 
$\alpha\in Q(S[1/a])$,  and $S[1/a]$ is normal, $\alpha\in S[1/a]$.
So $\alpha\in S[1/a]\cap Q(S)^G=S[1/a]^G=A[1/a]\subset Q(A)$, as desired.
\end{proof}

\begin{corollary}\label{Italian2.thm}
If $S$ is normal, then $Q(S^G)=Q(S)^G$ if and only if there exists some
closed point $x\in\Spec S$ such that $\dim Gx\geq r$.
\end{corollary}

\begin{proof}
As $s=\max\{\dim Gx\mid x\in X\}$, there exists some closed point $x
\in\Spec S$ such that $\dim Gx\geq r$ if and only if $s\geq r$.
As $r\geq s$ is always true, this is equivalent to say that $s=r$.
The assertion follows from Lemma~\ref{Italian.thm}.
\end{proof}

\section{Examples}

\paragraph\label{ex-1.par}
Let $n\geq m\geq t\geq 2$ be positive integers, $V=k^m$, $W=k^n$, 
and $M:=V\otimes W$.
Let $v_1,\ldots,v_m$ and $w_1,\ldots,w_n$ respectively be the standard
bases of $V$ and $W$.
Note that $\GL_m\times \GL_n=\GL(V)\times \GL(W)$ 
acts on $M$ in a natural way.
Let $U\subset\GL_m$ be the subgroup of unipotent upper triangular matrices.
Then $U$ is connected, and 
$\bar k\otimes_k U$ does not have a nontrivial character.
Note that $\Sym M=
k[x_{ij}]_{1\leq i\leq m,\;1\leq j\leq n}$ is a polynomial
ring in $mn$ variables, where $x_{ij}=v_i\otimes w_j$.
The ideal $I_t$ of $\Sym M$ generated by the $t$-minors of $(x_{ij})$ is
a $\GL_m$-stable ideal.
Set $S:=k[M]/I_t$.

The class group $\Cl(S)$ of $S$ is $\Bbb Z$, see \cite[(8.4)]{BV}.
So $S$ is not a UFD.
Let $P$ be the ideal generated by the $(t-1)$-minors of the first $(t-1)$ rows
of $(x_{ij})$.
Then $P$ is $U$-stable, and the class $\overline{\langle P\rangle}$ 
of $P$ in the class group
$\Cl(S)\cong\Bbb Z$ is a generator, see \cite[(8.4)]{BV}.

As $I_t$ is $\GL_m\times \GL_n$-stable, 
$\GL_m\times \GL_n$ acts on $S$ in a natural way.
By \cite[Theorem~9]{Grosshans}, $S^U$ is finitely generated over $k$.

Note that $S^U\cong\bigoplus_{\lambda}\nabla_{\GL_n}(\lambda)$
as a $\GL_n$-algebra, where $\lambda$ runs through
all the sequences $(\lambda_1,\ldots,\lambda_{t-1})$ such that
$\lambda_1\geq \cdots\geq \lambda_{t-1}\geq 0$, and $\nabla_{\GL_n}(\lambda)$ 
denotes the dual Weyl module of highest weight $\lambda$.
Letting $\nabla_{\GL_n}(\lambda)$ be of degree 
$\lambda_1+\cdots+\lambda_{t-1}$,
such a $\GL_n$-algebra has the same Hilbert function as the
Cox ring of $\GL_n/\Cal P$, where $\Cal P$ is the parabolic subgroup
of $(a_{ij})\in\GL_n$ with $a_{ij}=0$ for $j>i<t$.
The dimension of $\GL_n/\Cal P$ is $(t-1)(n-t/2)$, and the rank of the
class group of $\GL_n/\Cal P$ is $t-1$.
So $\dim S^U=(t-1)(n+1-t/2)$.

On the other hand, $P^U=I_{t-1}^U$, and hence it is easy to see that
$S^U/P^U\cong (S/I_{t-1})^U$, and its dimension is
$(t-2)(n+1-(t-1)/2)$.
So the height of $P^U$ is $n-t+2\geq 2$.

In order to apply Theorem~\ref{main.thm} to conclude that  $S^U$ is a UFD,
it remains to show that there is an orbit $U\cdot 
x$ of $\Spec S$ whose dimension
is equal to $\dim S-\dim S^U$ by Corollary~\ref{Italian2.thm}.
Note that $\dim S=(t-1)(m+n-t+1)$ by \cite[(1.1)]{BV}.
So $\dim S-\dim S^U=(t-1)(m-t/2)$.

Next, $\Spec\Sym M=M^*=V^*\otimes W^*\cong \Hom(W,V^*)\cong\Mat(m,n;k)$.
As matrices, $g\in\GL_m$ acts on $A\in M^*=\Mat(m,n;k)$ by
$g\cdot A={}^tg^{-1}A$, where the product on the right hand side is the
usual multiplication of matrices.

Set $x\in(\Spec S)(k)=\{A\in\Mat(m,n;k)\mid \rank A<t\}$ to be
the matrix
\[
x=
\begin{pmatrix}
E_{t-1} & O_{t-1,n-t+1} \\
O_{m-t+1,t-1} & O_{m-t+1,n-t+1}
\end{pmatrix},
\]
where $O_{i,j}$ is the $i\times j$ zero matrix, and $E_{t-1}$ is the
identity matrix of size $t-1$.

Then
\[
U\cdot x= U^-x=
\left\{
\begin{pmatrix}
U_{t-1}^- & O_{t-1,n-t+1}\\
* & O_{m-t+1,n-t+1}
\end{pmatrix}
\right\}
,
\]
where $U^-$ (resp.\ $U_{t-1}^-$) is the group of $m\times m$ 
(resp.\ $(t-1)\times (t-1)$) unipotent lower triangular matrices.
So $\dim U\cdot x=(t-1)(m-t/2)$, as desired.
This proves that $S^U$ is a UFD.

There is another way to show that $S^U$ is a UFD.
Some more argument shows that $S^U$ is isomorphic to the Cox ring of
$\GL_n/\P$ (we omit the proof).
We can invoke \cite[Corollary~1.2]{EKW}.

\paragraph\label{hochster.par}
Let $k$ be a field of characteristic $3$.
Let $G$ be the cyclic group $\Bbb Z/3\Bbb Z$ with the generator $\sigma$.
Let $\Lambda$ be the group algebra $\Bbb ZG$, and $M$ the left ideal
$\Lambda(1-\sigma)$ of $\Lambda$.
Then $G$ acts on the group algebra $S=kM$.
Let $A=1-\sigma$ and $B=\sigma(1-\sigma)$.
Then $S=k[A^{\pm 1},B^{\pm 1}]$ is the Laurent polynomial ring,
and $G$ acts on $S$ via $\sigma A=B$, and $\sigma B=(AB)^{-1}$.
So $G$ acts on 
$\Spec S\cong \Bbb A^2\setminus\{\text{the coordinate lines}\}$ via
$\sigma(x,y)=(1/xy,x)$.
Thus the only fixed point is $(1,1)$, and $G$ acts freely on $\Spec S
\setminus\{(1,1)\}$.
Thus the action of $G$ on $S$ is effective (i.e., $G\rightarrow\Aut S$ is
injective), and $\Spec S\rightarrow \Spec S^G$ is \'etale in codimension one.
As $S$ is a UFD, the class group of $S^G$ is $H^1(G,S^\times)$ by
\cite[(16.1)]{Fossum}.

Let $u\in S^\times$.
Then there exists some $(m,n)$ such that 
\[
A^mB^n u\in k[A,B]^\times=k^\times.
\]
So we have an isomorphism of 
$\Lambda$-modules $S^\times\cong k^\times\oplus M$.
Now we compute the cohomology group of $S$.
Take the resolution
\[
\Bbb F:\cdots\xrightarrow{1-\sigma}\Lambda
\xrightarrow{1+\sigma+\sigma^2}\Lambda
\xrightarrow{1-\sigma}\Lambda\rightarrow0
\]
of the trivial $\Lambda$-module $\Bbb Z$.
Then $\Hom_G(\Bbb F,N)$ is
\[
\cdots\xleftarrow{1-\sigma}N
\xleftarrow{1+\sigma+\sigma^2}N
\xleftarrow{1-\sigma}N\leftarrow 0
\]
for any $\Lambda$-module $N$.
If $N=k^\times$, then $1-\sigma$ on $N$ is the zero map, while
$1+\sigma+\sigma^2$ is the action of $3$.
Namely, the map $\alpha\mapsto \alpha^3$.
So $H^1(G,k^\times)=\Ker 3$ is trivial.
On the other hand, if $N=M$, then $1+\sigma+\sigma^2$ on $N$ is zero,
and $H^1(G,M)\cong M/M(1-\sigma)\cong \Bbb Z/3\Bbb Z$.
This shows that $\Cl(S^G)\cong H^1(G,S^\times)\cong \Bbb Z/3\Bbb Z$
is not trivial, and $S^G$ is not a UFD.

On the other hand, there is no nontrivial homomorphism $G\rightarrow S^\times
\cong k^\times\oplus M$, since $k$ is of characteristic $3$, and 
$M$ is torsion free.
% Thus this example disproves \cite[(6.2)]{Hochster}.
This example shows that the assumption {\bf (iii)} in 
Theorem~\ref{popov.thm} cannot be removed.

We compute $S^G$ more precisely.
Let $D$ be the group algebra $k\Lambda$.
Then $D=k[X_1^{\pm1},X_2^{\pm1},X_3^{\pm1}]$ is a Laurent polynomial ring
(here $X_1=1$, $X_2=\sigma$, and $X_3=\sigma^2$)
on which $G$ acts via $\sigma X_1=X_2$, $\sigma X_2=X_3$, and
$\sigma X_3=X_1$.
The subalgebra $C:=k[X_1,X_2,X_3]$ is a $G$-subalgebra, and
it is well-known that $C^G=k[e_1,e_2,e_3,\Delta]$, where 
$e_i$ is the $i$th elementary symmetric polynomial, and
$\Delta=(X_2-X_1)(X_3-X_1)(X_3-X_2)$.
So localizing by $e_3$, we have that
$D^G=k[e_1,e_2,e_3,\Delta,e_3^{-1}]$.
The action of $G$ on $D$ preserves the grading, where $\deg X_i=1$.
Then letting $A=X_1/X_2$ and $B=X_2/X_3$, $S$ is the degree zero component
of $D$.
Thus $S^G$ is the degree zero component of $D^G$.
It is $k[e_1^3/e_3, \Delta/e_3, e_2^3/e_3^2, e_1e_2/e_3]$.
Note that $e_1^3/e_3$ is an irreducible element in $S^G$.
It does not divide $e_1e_2/e_3$, 
but $(e_1^3/e_3)(e_2^3/e_3^2)=(e_1e_2/e_3)^3$.
This shows that $e_1^3/e_3$ is not a prime element.
This computation also shows that $S^G$ is not a UFD.

\paragraph\label{GL_2.par}
Let $k$ be a field.
Let $\bar k$ denote its algebraic closure.
Let $x\in\bar k\setminus k$ such that $t=x^2\in k$.
Let $G$ be the closed subset of $\GL_2(k)$
\[
G=\left\{
\begin{pmatrix}
a & b \\
c & d
\end{pmatrix}
\mid
ad-bc\neq 0,\;
a=d, b=tc
\right\}
\subset \GL_2(k).
\]
It is easy to see that $G$ is a closed subgroup of $\GL_2(k)$,
and is smooth over $k$ and connected.
Note also that $G$ is two-dimensional and commutative.
Let us write
\[
k[G]=k[A,B,C,D,(AD-BC)^{-1}]/(A-D,B-tC)\cong
k[A,C,(A^2-tC^2)^{-1}].
\]
Then it is easy to see that $f=A+xC$ is a group-like element of 
$S=K\otimes_k k[G]$, where $K=k(x)$.
$S$ is finitely generated over $k$ and is a domain.
With the right regular action, $S$ is a $G$-algebra, and
$\omega(f)=f\otimes A+xf\otimes C\in Sf\otimes_k k[G]$.
So $\omega_S(Sf)\subset Sf\otimes_k k[G]$, and $Sf$ is a $G$-ideal.
However, $f$ is not a semiinvariant.
It is a semiinvariant of $K\otimes_k G$.
Thus the assumption that $X(G)\rightarrow X(K\otimes_k G)$ is surjective
in Lemma~\ref{principal.thm} is really necessary.

\paragraph Let $k=\Bbb R$.
Then
\[
G=\left\{
\begin{pmatrix}
a & -b \\
b & a
\end{pmatrix}
\mid
a^2+b^2=1
\right\}
\subset \GL_2(k)
\]
is an anisotropic one-dimensional torus.
Let $G$ act on $S=\Bbb C[x,y,s,t]$ by
\begin{multline*}
\begin{pmatrix}
a & -b \\
b & a
\end{pmatrix}
x=(a+b\sqrt{-1})x,
\quad
\begin{pmatrix}
a & -b \\
b & a
\end{pmatrix}
y=(a+b\sqrt{-1})y,\\
\begin{pmatrix}
a & -b \\
b & a
\end{pmatrix}
s=(a-b\sqrt{-1})s,
\quad
\begin{pmatrix}
a & -b \\
b & a
\end{pmatrix}
t=(a-b\sqrt{-1})t
\end{multline*}
($G$ acts trivially on $\Bbb C$).
Then $S$ is a finitely generated UFD over $\Bbb R$,
$G$ is connected, $X(G)$ is trivial (however, $X(\Bbb C\otimes_{\Bbb R}G)$ 
is nontrivial), but $S^G=\Bbb C[xs,xt,ys,yt]$ is not a UFD.
The assumption that $X(G)\rightarrow X(K\otimes_k G)$ is surjective
in Corollary~\ref{ufd-cor.thm}
cannot be removed.

\end{document}